\documentclass[a4,reqno]{amsart}

\usepackage{esint} 
\usepackage{amsmath, amsthm}
\usepackage{mathrsfs}    
\usepackage{amssymb}
\usepackage{amsmath, amsthm}
\usepackage{mathrsfs}    
\usepackage{amssymb}
\usepackage{dsfont}

\usepackage{graphicx,color}

\usepackage{verbatim} 

\usepackage[colorlinks]{hyperref}
\hypersetup{linkcolor=black,citecolor=black,filecolor=black,urlcolor=blue}

\usepackage[showonlyrefs]{mathtools}

\usepackage[sort,nocompress]{cite}

\usepackage{tikz}
\usepackage{pgfplots}
\pgfplotsset{compat=1.12}

\newtheorem{theorem}{Theorem}[section]
\newtheorem{proposition}[theorem]{Proposition}
\newtheorem{corollary}[theorem]{Corollary}
\newtheorem{lemma}[theorem]{Lemma}

\theoremstyle{definition}

\newtheorem{remark}[theorem]{Remark}

\usepackage[
  hmarginratio={1:1},     
  vmarginratio={1:1},     
  textwidth=15cm,        
  textheight=20cm,  
  heightrounded,          
]{geometry}

\numberwithin{equation}{section}

\def\d{\delta}
\def\eps{\varepsilon}

\def\cD{\mathcal{D}}
\def\cB{\mathcal{B}}
\def\cA{\mathcal{A}}

\def\N{\mathbb{N}}

\def\R{\mathbb{R}}

\def\supp{\mathrm{supp\,}}
\def\dist{\mathrm{dist}}

\def\hbar{\bar{h}}

\def\phib{\boldsymbol{\phi}}

\def\Ical{\mathcal{I}}

\def\Omegae{\Omega_\eps}

\def\le{\leqslant}
\def\ge{\geqslant}
\def\le{\leqslant}

\renewcommand{\epsilon}{\varepsilon}
\renewcommand{\geq}{\geqslant}
\renewcommand{\leq}{\leqslant}

\newcommand{\rst}[1]{\ensuremath{{\mathbin |}%
\raise-.5ex\hbox{$#1$}}}
\newcommand{\mf}[1]{\mathbf{#1}}

\newcommand{\loc}{\mathrm{loc}}

\newcommand{\bs}[1]{\boldsymbol{#1}}

 \renewcommand{\(}{\left(}
\renewcommand{\)}{\right)}

\definecolor{verde}{rgb}{0,0.35,0.1} 
\definecolor{rosso}{rgb}{0.7,0,0}
\definecolor{blue}{rgb}{0,0,1}
\definecolor{viola}{rgb}{0.6,0,0.4}
\definecolor{grigio}{rgb}{0.5,0.5,0.5}

\def\sideremark#1{\ifvmode\leavevmode\fi\vadjust{\vbox to0pt{\vss
 \hbox to 0pt{\hskip\hsize\hskip1em
 \vbox{\hsize2.1cm\tiny\raggedright\pretolerance10000
  \noindent #1\hfill}\hss}\vbox to15pt{\vfil}\vss}}}%

\author[A. Pistoia]{Angela Pistoia}
\address{Angela Pistoia \newline \indent Dipartimento di Metodi e Modelli Matematici, Universit\`a di Roma ``La Sapienza''\newline \indent Via Antonio Scarpa 16, 00161 Roma (Italy)}
\email{angela.pistoia@uniroma1.it}

\author[N. Soave]{Nicola Soave}
\address{Nicola Soave \newline \indent
Dipartimento di Matematica,  Politecnico di Milano,  \newline \indent
Via Edoardo Bonardi 9, 20133 Milano (Italy)}
\email{nicola.soave@gmail.com; nicola.soave@polimi.it}

\author[H. Tavares]{Hugo Tavares}
\address{Hugo Tavares \newline \indent CMAFcIO \& Departamento de Matem\'atica\newline \indent Faculdade de Ci\^encias da Universidade de Lisboa \newline \indent Edif\'icio C6, Piso 1, Campo Grande 1749-016  Lisboa (Portugal)}
\email{hrtavares@ciencias.ulisboa.pt}

\thanks{\textit{Acknowlegments.}  A. Pistoia is partially supported by Sapienza research grant ``Nonlinear PDE's in geometry
and physics''. N. Soave is partially supported by the PRIN-2015KB9WPT$\_$010 Grant: ``Variational methods, with applications to problems in mathematical physics and geometry''. H. Tavares is partially supported by the Portuguese government through FCT - Funda\c c\~ao para a Ci\^encia e a Tecnologia, I.P., both under the project PTDC/MAT-PUR/28686/2017 and through the grant UID/MAT/04561/2013. The second and third author are also supported by the ERC Advanced Grant 2013 n. 339958 ``Complex Patterns for Strongly Interacting Dynamical Systems - COMPAT''}

\title[A fountain of positive Bubbles on a Coron's Problem for Gradient Systems]{A fountain of positive Bubbles on a Coron's Problem for a Competitive Weakly Coupled Gradient System}
\date{\today}

\begin{document}
\maketitle

\begin{abstract}
We consider the following critical elliptic system:
\begin{equation*}
\begin{cases}
-\Delta u_i=\mu_i u_i^{3}+\beta u_i^{ } \sum\limits_{j\neq i} u_j^{2} \quad  \hbox{in}\ \Omega_\eps \\
u_i=0 \hbox{ on } \partial\Omega_\eps , \qquad  u_i>0  \hbox{ in } \Omega_\eps 
\end{cases}\qquad i=1,\ldots, m,
\end{equation*}
in a domain $\Omega_\eps \subset \R^4$ with a small shrinking hole $B_\eps(\xi_0)$. For $\mu_i>0$, $\beta<0$, and $\eps>0$ small, we prove the existence of a non-synchronized solution which looks like a fountain of positive bubbles, i.e.   each component $u_i$ exhibits a towering blow-up around $\xi_0$ as $\eps \to 0$. The proof is based on the Ljapunov-Schmidt reduction method, and the velocity of concentration of each layer within a given tower is chosen in such a way that the interaction between bubbles of different components balance the interaction of the first bubble of each component with the boundary of the domain, and in addition is dominant  when compared with the interaction of two consecutive bubbles of the same component. 
\end{abstract}

\section{Introduction}

This paper deals with the existence of   solutions to the elliptic critical system
\begin{equation}\label{s}
\begin{cases}
-\Delta u_i=\mu_i u_i^{p}+\beta u_i^{ p-1\over 2} \sum\limits_{j\neq i} u_j^{p+1\over2} \quad  \hbox{in}\ \Omega \\
u_i=0 \hbox{ on } \partial\Omega , \qquad  u_i>0  \hbox{ in } \Omega 
\end{cases}\qquad i=1,\ldots, m,
\end{equation}
when $\Omega$ is a bounded smooth domain in $\mathbb R^N$, and $p=\frac{N+2}{N-2}=2^*-1$, with $2^*$ critical Sobolev exponent. Thinking at $u_i$ as a density function (which is natural since \eqref{s} is studied in connection with problems in nonlinear optics and Bose-Einstein condensation), the sign of the real parameters $\mu_i$ describes the self-interaction between particles of the same density $u_i$, and will always be positive: that is, we have attractive self-interaction. On the contrary, the coupling parameter $\beta$, which describes the interaction between particles of different densities, will always be negative: that is, we have repulsive mutual interaction. 

The system \eqref{s} has the trivial solution, i.e. all the components $u_i$ vanish. It can also have a semi-trivial solution, i.e. only $\ell<m$ components vanish. It is clear that in this case \eqref{s} reduces to a system with  $m-\ell$ nontrivial components, so we are naturally lead to find \emph{fully nontrivial solutions}, namely solutions where all the components are nontrivial. In fact, we will be concerned with \emph{positive solutions}, namely fully nontrivial solutions with $u_i>0$ for every $i$. 

It is useful to point out that \eqref{s} can have solutions with synchronized components, i.e. all the components satisfy $u_i=s_i u$ for some $s_i\in\mathbb R$ and $u$ solves the single equation
\begin{equation}
\label{e}
-\Delta u = u ^{p} \ \hbox{in}\ \Omega,\ u=0\
\hbox{on}\ \partial\Omega,\ u>0\ \hbox{in}\ \Omega.
\end{equation}
For instance, if the number of components is $m=2$, the space dimension is $N=4$ (so that $p=3$), and 
\[
- \sqrt{\mu_1 \mu_2} < \beta< \min\{\mu_1, \mu_2\} \quad \text{or} \quad \beta> \max\{\mu_1,\mu_2\},
\]
then a solution of \eqref{e} gives rise to a synchronized solution. In this way, results available for the single equation can be translated in terms of \eqref{s}: for instance, if $\Omega$ has nontrivial $\mathbb Z_2-$homology, then the celebrated Bahri-Coron's result \cite{BaCo} claims the existence of a positive solution for \eqref{e}, and in turn this gives existence of a synchronized solution for \eqref{s}. It is worthwile to recall also the Coron's result \cite{c}, where the case of a domain with a small hole has been considered, namely $\Omega$ is replaced by $\Omega_\epsilon:=\Omega\setminus \overline{B_\epsilon(\xi_0)}$, and problem \eqref{e} has a solution which blows-up at $\xi_0$ as $\epsilon\to0 $ (see also \cite{l,Rey_fr}). Again, this family of solutions can be used to construct an associated family of synchronized solutions for \eqref{s}.

The assumptions on the domain are natural, since, exactly as in the scalar case, a Pohozaev-type identity shows that there is no solution if $\Omega$ is starshaped (see for instance \cite[p. 519]{ChenZou1} or \cite{ChenZou2}). 

The above discussion induced the first two authors to investigate the following problem: does \eqref{s} have non-synchronized solutions? An affirmative answer is given in \cite{PistoiaSoave}, where \eqref{s} is posed in a domain $\Omega_\epsilon\subset \R^N$, with $N=3,4$, having $\kappa$ distinct holes; that is, $\Omega_\epsilon:=\Omega\setminus \cup_{i=1}^\kappa \overline{B_\eps(\xi_i)}$, with $2\le \kappa\le m$; for a quite general choice of interaction terms $\beta_{ij}$ (which can be both of cooperative type, and of competitive type), Pistoia and Soave proved existence and concentration results of solutions whose components are splitted in several groups $G_1, . . . , G_\kappa$, in such a way that each component within a given group $G_i$ concentrates around a point $\xi_i$ in a somehow synchronized fashion (in the sense that the velocity of concentration of different components belonging the same group is the same), while the different groups concentrate around different points. In particular, the main results in \cite{PistoiaSoave} regard the case when at least two components concentrate around different points, and hence cannot be synchronized.

In view of the above discussion, it is natural to ask the following question:
 {\it if the domain has only one small hole, is it still possible to find a non-synchronized solution?}
The main purpose of this paper is to give a positive answer for $\beta<0$ and $N=4$ - so that $p=3$ (for a discussion of the cases $N=3$ or other dimensions, see Remark \ref{rem:dimension} below). More precisely, we take
 \begin{equation}\label{eq:assOmega}
 \Omega\subset \R^4\text{ bounded domain, symmetric with respect to one of its points $\xi_0\in \Omega$,}
 \end{equation}
i.e. $x\in\Omega$ if and only if $2\xi_0-x\in\Omega$, and consider the following elliptic problem with $m\in \N$ equations:
\begin{equation}\label{eq:gensystem}
\begin{cases}
-\Delta u_i=\mu_i u_i^{3}+\beta u_i \sum\limits_{j\neq i} u_j^{2}\ \hbox{in}\ \Omega_\epsilon\\
u_i=0 \hbox{ on } \partial\Omega_\epsilon, \quad u_i>0  \hbox{ in } \Omega_\eps
\end{cases}\qquad i=1,\ldots, m,
\end{equation}
where $\Omega_\eps$ is a domain with one hole, $\Omega_\eps:=\Omega\setminus\overline{B_\eps(\xi_0)}\subset \R^4$,  and $B_\eps(\xi_0)$ denotes the open ball of $\R^4$ centered at $\xi_0$ with radius $\eps$. Throughout this paper we take $\mu_i>0$, the so called focusing case, and $\beta<0$, which means that the coupling terms in \eqref{eq:gensystem} are of competitive type. 
 
We find solutions of \eqref{eq:gensystem} which look like a fountain of bubbles, namely their components  are a superposition of bubbles centered at $\xi_0$ with different rates of concentration. In particular,  all the components have a towering blow-up point at $\xi_0.$ This new phenomena is quite surprising, since it is in sharp contrast with the case of the single equation for which positive solutions cannot have neither  clustering or  towering  blow-up points,  i.e.  at every blow-up point there is at most one bubble concentrating there  (see   Schoen \cite{s}). We also mention that it is somehow unexpected that in a competitive regime (with a possibly large $|\beta|$) we find solutions whose components concentrate at the same point; this is only possible because the concentration rates are different, and in particular such solutions are not synchronized.

In order to state our results we need to introduce some notations. We  define 
\begin{equation}\label{eq:bubble}
U_{\delta, \xi}= \alpha_4 \frac{\delta}{\delta^2+|x-\xi|^2},\ \delta>0,\ x,\xi\in\mathbb R^N
\end{equation}
 (a \emph{bubble})
  with $\alpha_4=2\sqrt{2}$: these functions are all the positive solutions of the problem
 $$
-\Delta U=U^3, \qquad U\in \mathcal{D}^{1,2}(\R^4).
$$
 (see \cite{Aubin,CaGiSp,Talenti}). Also, we denote by $P_\eps:\mathcal{D}^{1,2}(\R^4)\to H^1_0(\Omega_\eps)$ the projection map and we define the projection of the bubble defined in \eqref{eq:bubble} as $W:=P_\eps U_{\delta,\xi}\in H^1_0(\Omega_\eps)$, which is the unique solution of
\begin{equation}\label{eq:Projection}
-\Delta W=-\Delta U_{\delta,\xi}=U_{\delta,\xi}^3 \text{ in } \Omega_\eps,\qquad W=0 \text{ on } \partial \Omega_\eps.
\end{equation}

Take $k\in \N$ (the total number of bubbles) larger than or equal to $m$ . Consider $I_1,\ldots, I_m\subset \{1,\ldots, k\}$ satisfying the  following properties:
\begin{enumerate}
\item $1\in I_1$;
\item $I_i\neq \emptyset$ for every $i=1,\ldots, m$;
\item $I_i\cap I_j=\emptyset$ whenever $i\neq j$;
\item $I_1\cup \ldots \cup I_m=\{1,\ldots, k\}$;
\item for every $j\in \{1,\ldots, k\}$ and $i\in \{1,\ldots, m\}$, if $j\in I_i$ then $j-1,j+1\not\in I_i$.
\end{enumerate}
Observe that one considers condition (1) without loss of generality, simply to fix ideas and simplify some statements. Conditions (2)-(3)-(4) imply that $I_1,\ldots, I_m$ form a partition of $\{1,\ldots, k\}$, while condition (5) means that each set $I_i$ does not contain two consecutive integers. Our main result is the following

\begin{theorem}\label{thm:main}
Take $\Omega$ satisfying \eqref{eq:assOmega} and let $\mu_i>0$, $\beta<0$. For any integer $k\geq m$ and for every partition $I_1,\ldots, I_m$ of $\{1,\ldots, k\}$ satisfying (1)--(5), there exists $\epsilon_0>0$ such that for any $\epsilon\in(0,\epsilon_0)$ problem \eqref{eq:gensystem} has a solution (symmetric with respect to $\xi_0$) of the form
$$
u_{i,\eps}=\mu_i^{-\frac{1}{p-1}}\sum\limits_{j\in I_i}P_{\eps}U_{\delta_j^\eps,\xi_0}+\phi_i^\eps,\quad i=1,\ldots, m
$$
with
\begin{equation}\label{eq:rates_thm}
\delta_j^\eps=d_j^\eps\, \epsilon^{j\over k+1}\(\log{1\over\epsilon}\)^{{1\over2}-{j\over k+1}}\text{ for some } d_j^\eps\to d_j^*,\quad \  j=1,\dots,k
\end{equation}
for
\[
d_j^*=\Gamma^\frac{i}{2(k+1)}\(A^2 \tau(0)\)^{\frac{i}{2(k+1)}-\frac{1}{2}}\(\frac{|\beta| \alpha_4^4 |\mathbb{S}^3|}{k+1}\)^{\frac{1}{2}-\frac{i}{k+1}}
\]
(see the upcoming \eqref{eq:constantsAB} and \eqref{eq:constantsAB2} for the expressions of the constants $A,\Gamma$) and
\[
\|\phi_i^\eps\|_{H^1_0(\Omega_\eps)}\to 0 \text{ as } \eps\to0,\quad i=1,\ldots, m.
\]
\end{theorem}

\begin{remark} As stated in the theorem, each component of the solution, $u_{i,\eps}$, belongs to the space
\[
H_{\eps,\xi_0}=\{u\in H^1_0(\Omega_\eps):\ u(x)=u(2\xi_0-x)\ \forall x\in \Omega_\eps\}.
\]
Since also $U_{\delta,\xi}$ is symmetry with respect to $\xi_0$, then $P_\eps U_{\delta_j^\eps,\xi_0}\in H_{\eps,\xi_0}$, as well as the remainder terms $\phi_i^\eps$.
\end{remark}

In order to better explain our result, let us take a particular case of \eqref{eq:gensystem} and Theorem \ref{thm:main}: 
\begin{equation}\label{eq:particularcase1}
m=2,\qquad k\geq 2,\qquad \text{ and } N=4\quad \text{(so that $p=3$)}.
\end{equation}
and the following partition of $\{1,\ldots,k\}$:
\begin{equation}\label{eq:particularcase2}
I_1=\{\text{odd numbers between $1$ and $k$}\},\ I_2=\{\text{even numbers between $1$ and $k$}\}.
\end{equation}
Clearly, $I_1,I_2$ satisfies conditions (1)--(5), and it is actually the only admissible partition for $m=2$.
Problem \eqref{eq:gensystem} now reads as

\begin{equation}\label{1}
\begin{cases}
-\Delta u_1=\mu_1u_1^{3}+\beta u_1u_2^2 \ \hbox{in}\ \Omega_\epsilon,\\
-\Delta u_2=\mu_2u_2^{3} +\beta u_1^2u_2 \ \hbox{in}\ \Omega_\epsilon\\
u_1=u_2=0\  \hbox{on}\ \partial\Omega_\eps, \quad u_1,u_2>0  \hbox{ in } \Omega_\eps.\\
\end{cases}
\end{equation}
In this particular situation, Theorem \ref{thm:main}  can be stated in the following way.
\begin{theorem}\label{thm:main2}
Take $\Omega$ satisfying \eqref{eq:assOmega} and  let $\mu_1,\mu_2>0$, $\beta<0.$ For any integer $k\ge2$, let $I_1,I_2$ be respectively the set of all odd and even numbers between 1 and k, as in \eqref{eq:particularcase2}. Then there exists $\epsilon_0>0$ such that for any $\epsilon\in(0,\epsilon_0)$ problem \eqref{1} has a solution (symmetric with respect to $\xi_0$) of the form
$$
u_{1,\eps}=\mu_1^{-\frac{1}{2}}\sum\limits_{j\in I_1} P_{\eps}U_{\delta_j^\eps,\xi_0}+\phi_1^\eps \quad \text{and}\quad
u_{2,\eps}=\mu_2^{-\frac{1}{2}}\sum\limits_{j\in I_2}P_{\eps}U_{\delta_j^\eps,\xi_0} +\phi_2^\eps
$$
with
\begin{equation}\label{eq:rates_thm}
\delta_j^\eps=d_j^\eps\, \epsilon^{j\over k+1}\(\log{1\over\epsilon}\)^{{1\over2}-{j\over k+1}}\text{ for some } d_j^\eps\to d_j^*,\quad \  j=1,\dots,k
\end{equation}
for
\[
d_j^*=\Gamma^\frac{j}{2(k+1)}\(A^2 \tau(0)\)^{\frac{j}{2(k+1)}-\frac{1}{2}}\(\frac{|\beta| \alpha_4^4 |\mathbb{S}^3|}{k+1}\)^{\frac{1}{2}-\frac{j}{k+1}}
\]
and
\[
\|\phi_1^\eps\|_{H^1_0(\Omega_\eps)}\to 0,\  \|\phi_2^\eps\|_{H^1_0(\Omega_\eps)} \to 0\quad \text{ as } \eps\to0.
\]
\end{theorem}

In order to avoid insignificant technicalities that would make the presentation harder to follow, we will simply prove Theorem \ref{thm:main2}; in order to convince the reader that the proof of Theorem \ref{eq:gensystem} follows precisely in the same way we will make some remarks along the paper (see Remarks \ref{rem:section2}, \ref{rem:section3}, \ref{rem:remark_section_reduced} and \ref{rem:section_conclusion}).

Our result is inspired by the construction performed by Musso and Pistoia in \cite{MussoPistoiaJMPA2006} and Ge, Musso and Pistoia in \cite{GeMuPi}, where the authors built sign-changing solutions to Coron's problem  whose shape resembles a superposition of bubbles centered at the point $\xi_0$ with alternating sign and with different rate of concentration. The proof here also follows the same scheme which is based on a Ljapunov-Schmidt procedure: we find a good first order approximation term (see \eqref{def L}),  we perform a linear theory for the linearized system around the ansatz (see Proposition \ref{lem: linear part}), we  reduce the problem to a finite dimensional one (see Proposition \ref{prop:C^1phi}) and finally we study the reduced problem (see Section \ref{sec:reducedenergy}).   However, the main steps of  our  proof  require rather delicate and careful estimates, see for instance the estimates involving the interacting term in the study of the linear part in Subsection \ref{subsec:linear}, the asymptotic expansion of the interaction energy (Lemma \ref{lemma:interaction_estimate}), and the estimate of the remainder term in Lemma \ref{lemma:remainder}. Indeed, the interaction between bubbles of different components has to balance the interaction of the first bubble of each component with the boundary of the domain, and most of all  it has to be dominant  compared with the interaction of two consecutive bubbles of the same component. Actually, this is possible because of   the presence of an $|\log\eps|$-order term which turns out to be crucial in our construction (see estimate \eqref{eq:crucial_ln}).
 
 \begin{remark} We prove the existence of   solutions which look  like    fountains of positive bubble all centered at the point $\xi_0$ when $\Omega$ is symmetric with respect $\xi_0.$  It is clear that using the same arguments of Ge, Musso and Pistoia \cite{GeMuPi}   we can remove the symmetry assumption, just centering all the bubbles $U_{\delta_i,\xi_i}$ at   suitable points $\xi_i=\xi_i( \epsilon)$ which approach $ \xi_0$ with a suitable rate as $\epsilon\to0.$ 
\end{remark}

\begin{remark}
For the sake of completeness, we also mention some recent results concerning the existence of solutions to system \eqref{s} when $\Omega$ is the whole space $\mathbb R^N$.
As far as we know, all the results deal with systems with only two components.
Guo, Li
and Wei in \cite{GuLiWe} established
the existence of  infinitely many positive nonradial  solutions of \eqref{s},  only when $N=3,$ in the competitive case. Peng, Peng and Wang discussed in \cite{PePeWa} uniqueness of the least energy solution for $\beta>0$, and the non-degeneracy of the manifold of the synchronized positive solutions. Clapp and Pistoia in \cite{cp} proved  
that  system \eqref{s} in any dimension  has infinitely many fully nontrivial solutions, which are not conformally equivalent.
Gladiali, Grossi and Troestler in \cite{ggt1,ggt2} obtained radial and nonradial solutions to some critical systems like \eqref{s}
using bifurcation methods. 
\end{remark}

\begin{remark} A Brezis-Nirenberg type problem has been studied for systems, see for instance \cite{ChenZou1,ChenZou2,ChenLinZouCPDE2014} for existence results, while for concentration and blow-up type results see \cite{ChenLin,PistoiaTavares}.
\end{remark}

\begin{remark}
 As already mentioned, appropriate assumptions on $\beta$ allows to obtain a synchronized solution to \eqref{s} if $\Omega$ has nontrivial $\mathbb Z_2-$homology. We conjecture that system \eqref{s} has at least one (actually we would say infinitely many) positive non-synchronized solution if $\Omega$ has nontrivial $\mathbb Z_2-$homology (as in Bahri-Coron's result for the single equation \eqref{e}) and $\beta<0$ is arbitrary. A  first attempt in this direction is due to 
 Clapp and Faya \cite{cf}, who establish the existence of a prescribed number of fully nontrivial
solutions to the system with only two components under suitable symmetry assumptions on the topologically nontrival domain $\Omega.$

We would like to remark that the difficulty in finding positive solutions to system \eqref{s}, even with only two components, is similar to the difficulty in finding sign-changing solutions for the single equation \eqref{e}.
One key point is the blow-up analysis of solutions: in the case of positive solutions the blow-up, whenever it occurs, is isolated and simple, while in the case of sign-changing solution   multiple bubbling
naturally appears.
\end{remark}

Without loss of generality, we will work from now on with 
\begin{equation}\label{eq:assumption_wlog}
\mu_1=\mu_2=1,\quad \text{ and take }\quad \xi_0=0\in \Omega,
\end{equation} assuming that $\Omega$ is symmetric with respect to the origin. Observe that we are conduced to such situation by eventually replacing $u_i$ with $\mu_i^{-\frac{1}{2}}u_i(x+\xi_0)$.

\begin{remark}Solutions of \eqref{eq:gensystem} correspond to critical points with nontrivial components of the $C^1$--energy functional $J_\eps:H^1_0(\Omega;\R^m)\to \R$ defined by
\[
J_\eps(u_1,\ldots, u_m)=\sum_{i=1}^{m}\int_{\Omega_\eps} \left(\frac{|\nabla u_i|^2}{2}-\frac{\mu_i(u_i^+)^{p+1}}{p+1}\right)-\frac{2\beta }{p+1}\mathop{\sum_{i,j=1}}_{i<j}^m\int_{\Omega_{\eps}} |u_i|^\frac{p+1}{2}|u_j|^\frac{p+1}{2}.
\]
Indeed, if $(u_1,\ldots, u_m)$ is a critical point of $J_\eps$, then it satisfies
\[
-\Delta u_i=\mu_i(u_i^+)^p+\beta \sum_{j\neq i} u_i |u_i|^\frac{p-3}{2}|u_j|^\frac{p+1}{2},\quad i=1,\ldots, m.
\]
Multiplying this equation by $u_i^-$ and integrating by parts yields (since $\beta<0$)
\[
0\geq -\int_{\Omega_\eps} |\nabla u_i^-|^2=-\beta \sum_{j\neq i}\int_{\Omega_\eps} |u_i^-|^\frac{p+1}{2}|u_j|^\frac{p+1}{2}\geq 0.
\]
If $u_i\not\equiv 0$, then by the maximum principle we deduce that $u_i>0$.
\end{remark}

\begin{remark}\label{rem:dimension} The Sobolev critical exponent is defined only for $N\geq 3$. On the other hand, for $p$ defined as before, the right hand sides of \eqref{s} are $C^1$ nonlinearities if and only if we have $\frac{p-1}{2}\geq 1$, if and only if $N\leq 4$. Therefore, it is reasonable to work in dimension $N=3$ or $N=4$. Here we chose to deal with the case $N=4$ only since it requires less technicalities: all the exponents are positive integers, which makes some expansions explicit.  Using Taylor expansions we could have takled the case $N=3$. We conjecture that in this case the main results (and in particular the rates) would be the same.
\end{remark}

\begin{remark} A similar approach could also be used to find solutions for   critical systems
in pierced domains when the interaction term is  more in general  like (e.g. Lotka-Volterra systems)
\begin{equation} 
\begin{cases}
-\Delta u_i=\mu_i u_i^{p}+\beta_i u_i^{q_i } \sum\limits_{j\neq i} u_j^{q_j }\ \hbox{in}\ \Omega_\eps \\
u_i=0 \hbox{ on } \partial\Omega_\eps , \quad u_i>0  \hbox{ in } \Omega _\eps
\end{cases}\qquad i=1,\ldots, m,
\end{equation}
when $\mu_i>0,$ $\beta_i<0$ and $q_i,q_j>1.$ In the non-variational cases, one has to replace the asymptotic estimates on the energy of Section \ref{sec:reducedenergy} with an argument that simply uses the system like in \cite[Section 2]{MussoPistoiaIndiana2002}.  
\end{remark}

\subsection*{Notations} 
Working with dimension $N=4$, we deal with the following bubbles concentrated at the origin
\[
U_{\delta,0}(x)=\alpha_4\frac{\delta}{\delta^2+|x|^2}
\]
(where $\alpha_4=2\sqrt{2}$), which we denote also by $U_{\delta}$; in many cases we deal with different concentration parameters $\delta_i$, $i=1,\dots,k$, and we shall simply write $U_{\delta_i}=U_i$. These correspond to all positive solutions of $-\Delta U=U^3$ in $\R^4$ which are symmetric with respect to the origin. It is well known (see \cite{BianchiEgnell}) that the space of solutions of the linearized equation
\begin{equation}\label{eq:linearized_eq}
-\Delta V=3U_{\delta}^{2}V
\end{equation}
has dimension $4+1=5$ in $\mathcal{D}^{1,2}(\R^5)$, being spanned by 
\[
\frac{\partial U_{\delta}}{\partial \delta}(x)= \alpha_4\frac{|x|^2-\delta^2}{(\delta^2+|x|^2)^2},\quad \frac{\partial U_{\delta}}{\partial \xi_i}(x)=2\alpha_4\frac{\delta x_i}{(\delta^2+|x|^2)^2},\ i=1,\ldots, 4.
\] Therefore, the space of solutions to \eqref{eq:linearized_eq} which belong to 
\[
\mathcal{D}^{1,2}_s(\R^4):=\{\psi\in \mathcal{D}^{1,2}(\R^4):\ \psi(-x)=\psi(x)\ \forall x\in \R^4\}
\]
has dimension 1, being spanned by $\frac{\partial U_\delta}{\partial \delta}$. For future convenience, we observe that
\begin{equation}\label{pa U con U}
\left|\frac{\partial U_\delta}{\partial \delta}(x)\right|  \le  \frac{U_\delta(x)}{\delta}.
\end{equation}

We take the following inner product and norm in $H^1_0(\Omega_\eps)$:
\[
\langle u,v\rangle_{H^1_0}:=\int_{\Omega_\eps} \nabla u\cdot \nabla v,\qquad \|u\|_{H^1_0}^2=\int_{\Omega_\eps} |\nabla u|^2
\]
and the standard $L^p$ norm by $\|\cdot \|_p$ (we omit the dependence on $\eps$ for simplicity).

The Green function of the Laplace operator in $\Omega$ with Dirichlet boundary conditions is denoted by $G(x,y)$, and can be decomposed as
$$
G(x,y)=\frac{\gamma_4}{|x-y|^2}-H(x,y),
$$
where $\gamma_4:=(2|\partial B_1|)^{-1}$, and  $H$ is the regular part of $G$ which, for every $x\in \Omega$, satisfies
$$
\begin{cases}
-\Delta_y H(x,y)=0 & \text{ for } y\in \Omega,\\
H(x,y)=\frac{\gamma_4}{|x-y|^{2}}   &  \text{ for } y\in \partial \Omega.
\end{cases}
$$
The Robin function of $\Omega$ is defined as $\tau(x):=H(x,x)$, and satisfies $\tau(x)\to +\infty$ as $\dist(x,\partial \Omega)\to 0$. Throughout the paper, we will always label the following constants: 
\begin{equation}\label{eq:constantsAB}
  A:=\int_{\R^4} U_{1,0}^3=\int_{\R^4} \frac{\alpha_4^3}{(1+|y|^2)^3}\, dy, \quad B:=\int_{\R^4}  U_{1,0}^4= \int_{\R^4}\frac{\alpha_4^4}{(1+|y|^2)^4}\, dy,
\end{equation}
\begin{equation}\label{eq:constantsAB2}
\Gamma:=\int_{\R^N} \frac{\alpha_4^4}{|y|^2(1+|y|^2)^3}\, dy,  
\end{equation}
and use $B_\eps,\partial B_\eps$ instead of $B_\eps(0),\partial B_\eps(0)$ respectively. We will denote the $L^p(\Omega_\eps)$ norms by  $\|\cdot \|_{L^p}$, while $\|u\|_{H^1_0}^2:=\int_{\Omega_\eps} |\nabla u|^2$ for every $u\in H^1_0(\Omega_\eps)$.
 
 \section{The ansatz and reduction scheme}
Recall that, without loss of generality, we assume \eqref{eq:assumption_wlog}; due to the symmetry, by the principle of symmetric criticality we can work in the space
\begin{equation}\label{eq:Heps}
H_\eps:=H_{\eps,0}=\{u\in H^1_0(\Omega_\eps):\ u(-x)=u(x)\ \forall x\in \Omega_\eps\}. 
\end{equation}  
 We deal with solutions of
\begin{equation}\label{eq:systemwithf}
\begin{cases}
-\Delta u_1=f(u_1)+\beta u_1u_2^2\\
-\Delta u_2=f(u_2)+\beta u_2u_1^2\\
u_1,u_2\in H^1_0(\Omega_\eps),
\end{cases}
\end{equation}
where $f:\R\to \R$, $f(s):=(s^+)^3$. 
Denote by $\Ical^*:L^{\frac{4}{3}}(\Omegae)\to H^1_0(\Omegae)$ the adjoint operator of the canonical Sobolev embedding $\Ical: H^1_0(\Omegae)\to L^4(\Omegae)$. This means that $v:=\Ical^* u$ can be defined as the (unique) weak solution of
\[
-\Delta v=u \text{ in } \Omega_\eps,\qquad v=0 \text{ on } \partial \Omega_\eps.
\]
Observe that, if $u$ is symmetric with respect to the origin, so is $\Ical^*u$.
The operator $\Ical^*$ is continuous: there exists $C>0$, independent of $\eps$, such that
\[
\|\Ical^*u\|_{H^1_0}\leq C \|u\|_{L^\frac{4}{3}} \qquad \forall u\in L^\frac{4}{3}(\Omegae).
\]
Using this operator, we can rewrite \eqref{eq:systemwithf} as
\begin{equation}\label{eq:adjointeq}
u_1=\Ical^*\left(f(u_1)+\beta u_1 u_2^2\right),\qquad u_2=\Ical^*\left(f(u_2)+\beta  u_2u_1^2 \right).
 \end{equation}
 Denote $U_j:=U_{\delta_j}$ for $j=1,\ldots, k$. Our ansatz is the following: for any integer $k\ge2$, we look for a solution of \eqref{eq:systemwithf} in $H_\eps$ of the form
\begin{equation}\label{eq:ansatz}
u_1=\sum\limits_{j\in I_1}P_{\eps}U_{j}+\phi_1\quad  \hbox{and}\quad 
u_2=\sum\limits_{j\in I_2} P_{\eps}U_{j}+\phi_2,
\end{equation}
where
\begin{equation}\label{eq:rates}
\delta_j=d_j\epsilon^{j\over k+1}\(\log{1\over\epsilon}\)^{{1\over2}-{j\over k+1}},\quad j=1,\ldots, k,
\end{equation} 
$\mathbf{d}=(d_1,\ldots, d_k)$ belongs to the set
\begin{equation}\label{eq:Xeta}
X_\eta=\left\{ \mathbf{d}\in \R^k:\  \eta<d_1,\ldots, d_k<1/\eta\right\}\quad \text{ for some $\eta \ll1$},
\end{equation}
and $\phi_1,\phi_2\in H_\eps$. 
\begin{remark}\label{rem:ratesrelations}
For future reference, we collect in this remark several important relations between the different rates $\delta_j$. Given $\eta>0$, we have
 \begin{equation}\label{delta}
\frac{\eps}{\delta_j}= \frac{1}{d_j} \eps^\frac{k+1-j}{k+1}\left(\log \frac{1}{\eps}\right)^{\frac{j}{k+1}-\frac{1}{2}}\to 0 \quad \text{ and } \quad {\delta_{j+1}\over\delta_j}=\frac{d_{j+1}}{d_j}\eps^\frac{1}{k+1}\left(\log \frac{1}{\eps}\right)^{-\frac{1}{k+1}}  \to 0
 \end{equation}
 as $\eps\to 0$, uniformly for $\mathbf{d}\in X_\eta$.
\end{remark}

For each $\eps>0$ small, our aim is to find $\eta>0$, $\mathbf{d}\in X_\eta$ and $\phi_1,\phi_2\in H_\eps$ such that, for $i,j=1,2$, $i\neq j$,
\begin{equation}\label{eq:systemajoint}
\sum\limits_{l\in I_i}P_{\eps}U_{l}+\phi_i=\Ical^*\left(f(\sum\limits_{l\in I_i}P_{\eps}U_{l}+\phi_i)+\beta (\sum\limits_{l\in I_i}P_{\eps}U_{l}+\phi_i)(\sum\limits_{l\in I_j}P_{\eps}U_{l}+\phi_j)^2\right).
\end{equation}
Given $\eps>0$ and $d_1,\ldots, d_k>0$, for $\delta_i$ defined as before define 
\[
\psi_{i}(x):=\frac{\partial U_{i}}{\partial {\delta_i}}(x)= \alpha_4\frac{|x|^2-\delta_i^2}{(\delta_i^2+|x|^2)^2}
\]
(recall the Notation section) and
\[
K_1 = K_{1,\mf{d},\eps}:=\text{span}\left\{P_\eps \psi_{j}:\ j\in I_1\right\},\quad K_2=K_{2, \mf{d},\eps}:=\text{span}\left\{P_\eps \psi_{j}:\ j\in I_2\right\},\quad \mf{K}_{\mf{d},\eps}:=K_1\times K_2.
\]
Observe that $\mf{K}_{\mf{d},\eps}^\perp=K_1^\perp\times K_2^\perp$. Moreover, consider the projection maps
\[
\Pi_i:H_\eps \to K_i,\qquad \Pi_i^\perp:H_\eps\to K_i^\perp,\qquad i=1,2.
\]
We can rewrite \eqref{eq:systemajoint} as a system of 4 equations: for $i,j=1,2$, $j\neq i$,
\begin{equation}\label{eq:systemajointProjection1}
\Pi_i\left(\sum\limits_{l\in I_i}P_{\eps}U_{l}+\phi_i\right)=\Pi_i\circ\Ical^*\left(f(\sum\limits_{l\in I_i}P_{\eps}U_{l}  +\phi_i) +\beta (\sum\limits_{l\in I_i}P_{\eps}U_{l}(x)+\phi_i)(\sum\limits_{l\in I_{j}}P_{\eps}U_{l}+\phi_{j})^2\right),
\end{equation}
\begin{equation}\label{eq:systemajointProjection2}
\Pi_i^\perp\left(\sum\limits_{l\in I_i}P_{\eps}U_{l}+\phi_i\right)=\Pi_i^\perp\circ\Ical^*\left(f(\sum\limits_{l\in I_i}P_{\eps}U_{l}  +\phi_i)+\beta (\sum\limits_{l\in I_i}P_{\eps}U_{l}+\phi_i)(\sum\limits_{l\in I_{j}}P_{\eps}U_{l}(x)+\phi_{j})^2\right).
\end{equation} 
In the next section, given $\eps,\eta>0$ sufficiently small and $\mf{d}\in X_\eta$, we find a unique $(\phi_1,\phi_2)=(\phi_1^{\mf{d},\eps},\phi_2^{\mf{d},\eps})\in K_{\mf{d},\eps}^\perp$ solution to \eqref{eq:systemajointProjection2}. By plugging this result in \eqref{eq:systemajointProjection1}, we end up having a problem with unknown $\mf{d}\in\R^k$ (thus a finite dimensional problem), which can be stated in terms of a \emph{reduced} energy. We analyse this reduced energy in Section \ref{sec:reducedenergy}.

\begin{remark}\label{rem:section2}
For the general system \eqref{eq:gensystem} and given a partition $I_1,\ldots, I_m$ of $\{1,\ldots, k\}$, the ansatz is exactly the same: $u_i=\sum_{j\in I_i} U_i+\phi_i$, for $i=1,\ldots, m$, where $\phi_i\in K_i$. We denote in this case $\mf{K}_{\mf{d},\eps}=K_1^\perp \times \ldots \times K_m^\perp$, and split the system of $m$ equations:
\begin{equation}\label{eq:systemajoint}
\sum\limits_{l\in I_i}P_{\eps}U_{l}+\phi_i=\Ical^*\left(f(\sum\limits_{l\in I_i}P_{\eps}U_{l}+\phi_i)+\beta (\sum\limits_{l\in I_i}P_{\eps}U_{l}+\phi_i)\mathop{\sum_{j=1}^m}_{j\neq i}(\sum\limits_{l\in I_j}P_{\eps}U_{l}+\phi_j)^2\right)
\end{equation}
($i=1,\ldots, m$) in $2m$ equations using the projection maps $\Pi_i$ and $\Pi_i^\perp$.
\end{remark}

\section{Reduction to a Finite Dimensional Problem}\label{sec:reduction}

In this section we study the solvability of \eqref{eq:systemajointProjection2}. We rewrite \eqref{eq:systemajointProjection2} as
\begin{equation}\label{L=N+R}
L_{\mf{d},\eps}^i(\bs{\phi}) = N_{\mf{d},\eps}^i(\bs{\phi}) + R_{\mf{d},\eps}^i,
\end{equation}
where $L$ stays for the linear part  
\begin{equation}\label{def L}
\begin{split}
L_{\mf{d},\eps}^1(\bs{\phi}) & = \Pi_1^{\perp} \Bigg\{ \phi_1 -\Ical^*\Bigg[ f'(\sum\limits_{j\in I_1}P_{\eps}U_{j}) \phi_1 + \beta (\sum\limits_{j\in I_2}P_{\eps}U_{j})^2 \phi_1  \\
& \hphantom{=\Pi_1^{\perp} \Bigg\{ \phi_i} + 2\beta(\sum_{j \in I_1} P_\eps U_{j})(\sum_{j \in I_2} P_\eps U_{j}) \phi_2\Bigg]\Bigg\},
\end{split}
\end{equation} 
$N$ stays for the nonlinear part
\begin{equation}\label{def N}
\begin{split}
N_{\mf{d},\eps}^1 &(\bs{\phi}) = \Pi_1^{\perp} \circ \Ical^* \left[ f( \sum_{j \in I_1} P_\eps U_{j}+\phi_1) - f( \sum_{j \in I_1} P_\eps U_{j}) - f'( \sum_{j \in I_1} P_\eps U_{j})\phi_1 \right.\\
& \hphantom{(\bs{\phi}) = \Pi_i^{\perp} \circ \Ical^* \Bigg[} + \beta ( \sum_{j \in I_1} P_\eps U_{j}+\phi_1)( \sum_{j \in I_2} P_\eps U_{j}+\phi_2)^2 - \beta ( \sum_{j \in I_1} P_\eps U_{j}) ( \sum_{j \in I_2} P_\eps U_{j})^2 \\
& \left.\hphantom{(\bs{\phi}) = \Pi_i^{\perp} \circ \Ical^* \Bigg[} - \beta (\sum\limits_{j\in I_2}P_{\eps}U_{j})^2 \phi_1 - 2\beta(\sum_{j \in I_1} P_\eps U_{j})(\sum_{j \in I_2} P_\eps U_{j}) \phi_2 \right]\\
&= \Pi_1^{\perp} \circ \Ical^* \left[ f( \sum_{j \in I_1} P_\eps U_{j}+\phi_1) - f( \sum_{j \in I_1} P_\eps U_{j}) - f'( \sum_{j \in I_1} P_\eps U_{j})\phi_1 \right.\\
&\left. \hphantom{(\bs{\phi}) = \Pi_i^{\perp} \circ \Ical^* \Bigg[} + \beta  ( \sum_{j \in I_1} P_\eps U_{j}) \phi_2 ^2 + 2\beta (\sum_{j\in I_2} P_\eps U_j)\phi_1\phi_2+\beta\phi_1\phi_2^2 \right],
\end{split}
\end{equation}
and $R$ is the remainder term
\begin{equation}\label{def R}
\begin{split}
R_{\mf{d},\eps}^1 & = \Pi_1^{\perp} \Bigg\{ - \sum_{j \in I_1} P_\eps U_{j}   + \Ical^*\Bigg[ f(\sum_{j \in I_1} P_\eps U_{j}) + \beta (\sum_{j \in I_1} P_\eps U_{j})(\sum_{j \in I_2} P_\eps U_j)^2\Bigg]\Bigg\} \\
& = \Pi_1^{\perp} \circ \Ical^* \Bigg[ f(\sum_{j \in I_1} P_\eps U_{j})- \sum_{j \in I_1} f(  U_{j}) + \beta (\sum_{j \in I_1} P_\eps U_{j})(\sum_{j \in I_2} P_\eps U_j)^2\Bigg]
\end{split}
\end{equation}
where the last equality is a consequence of the definitions of $\Ical^*$ and of $f$ (analogue expressions hold for $L^2_{\mf{d},\eps}$, $N^2_{\mf{d},\eps}$ and $R^2_{\mf{d},\eps}$).

We also define 
\[
\mf{L}_{\mf{d},\eps} := (L_{\mf{d},\eps}^1,L_{\mf{d},\eps}^2): \mf{K}_{\mf{d},\eps}^\perp \to \mf{K}_{\mf{d},\eps}^\perp,
\]
and $\mf{R}_{\mf{d},\eps}$ and $\mf{N}_{\mf{d},\eps} $ in an analogue way.

\begin{proposition}\label{prop:C^1phi}
Let $\beta<0$. Then for every $\eta>0$ sufficiently small there exists $\eps_0>0$ and $C>0$ such that, whenever $\eps\in (0,\eps_0)$ and $\mf{d}\in X_\eta$, there exists a unique function $\phib=\phib^{\mf{d},\eps}\in \mf{K}^\perp_{\mf{d},\eps}$ solving the equation
\begin{equation*}
\mf{L}_{\mf{d},\eps}(\phib)=\mf{R}_{\mf{d},\eps}+\mf{N}_{\mf{d},\eps}(\phib).
\end{equation*}
and satisfying
\begin{equation*}
\|\phib^{\mf{d},\eps}\|_{H^1_0(\Omegae)}\leq C \eps^\frac{1}{k+1} \left(\log\left(\frac{1}{\eps}\right)\right)^{-\frac{1}{k+1}} =\text{o}(\delta_1) 
\end{equation*}
 Moreover, the map  $
 X_\eta \to \mf{K}^\perp_{\mf{d},\eps},$ $\mf{d} \mapsto \phib^{\mf{d},\eps}
$
is of class $\mathcal{C}^1.$

\end{proposition}

The proof of the proposition takes the rest of this section, and is divided into several intermediate lemmas. 

\subsection{Study of the linear part}\label{subsec:linear}

As a first step, it is important to understand the solvability of the linear problem associated with \eqref{L=N+R}, i.e. 
\begin{equation}\label{linear pb}
L_{\mf{d},\eps}^i(\bs{\phi}) = f_i, \quad \text{with} \quad f_i \in K_{i}^\perp.
\end{equation}

\begin{proposition}\label{lem: linear part}
For every $\eta>0$ small enough there exists $\eps_0>0$ small, and $C>0$, such that if $\eps \in (0,\eps_0)$ then
\begin{equation}\label{17ott1}
\|\mf{L}_{\mf{d},\bs{\tau},\eps}(\bs{\phi})\|_{H_0^1(\Omega_\eps)} \ge C \|\bs{\phi}\|_{H_0^1(\Omega_\eps)}  \qquad \forall \bs{\phi} \in H_0^1(\Omega_\eps,\R^2)
\end{equation}
for every $\mf{d} \in X_\eta$. Moreover, $\mf{L}_{\mf{d},\eps}$ is invertible in $\mf{K}_{\mf{d},\eps}^\perp$, with continuous inverse.
\end{proposition}

The long proof proceeds by contradiction. For a fixed $\eta>0$ small, let us suppose that there exist sequences
\[
\{\eps_n\} \subset \R^+, \ \eps_n \to 0, \  \{\mf{d}_n\} \subset X_\eta,\  \{\bs{\phi}_n\} \subset K_{1,n}^\perp \times  K_{2,n}^{\perp}
\]
such that
\[
\| \bs{\phi}_n\|_{H_0^1(\Omega_{\eps_n})}=1 \quad \text{and} \quad \|\mf{L}_{n}(\bs{\phi}_n)\|_{H_0^1(\Omega_{\eps_n})} \to 0
\]
as $n \to \infty$, where we wrote $K_{i,n}:= K_{i, \mf{d}_{n},\eps_n}$ and $\mf{L}_n:= \mf{L}_{\mf{d}_n,  \eps_n}$ for short. In the same spirit, in this proof we write $P_n:= P_{\eps_n}$, $U_{i,n}:= U_{\d_{i,n},0}$, $ \psi_{i,n}:= \psi_{\d_{i,n},0}$, and $\Omega_n:= \Omega_{\eps_n}$.

Let $\mf{h}_n:= \mf{L}_{n}(\bs{\phi}_n)$. Then, by definition of $\mf{L}_n$,
\begin{equation}\label{eq phi}
\begin{split}
 \phi_{1,n}   & = h_{1,n} + w_{1,n} \\
& + \Ical^*\Bigg[ 3(\sum_{j \in I_1}P_n U_{j,n})^{2} \phi_{1,n}   + \beta (\sum_{j \in I_2} P_n U_{j,n})^2  \phi_{1,n} + 2\beta
(\sum_{j \in I_1} P_n U_{j,n}) (\sum_{j \in I_2} P_n U_{j,n})  \phi_{2,n} \Bigg] \end{split}
\end{equation}
(an analogue equation holds for $\phi_{2,n}$) for some $w_{i,n} \in K_{i,n}$.

\begin{lemma}
$\|w_{i,n}\|_{H_0^1(\Omega_n)} \to 0$ as $n \to \infty$. 
\end{lemma}
\begin{proof}
We focus on $w_{1,n}$, the proof for $w_{2,n}$ is analogue. As $w_{1,n} \in K_{1,n} = \textrm{span}\{P_n \psi_{j,n}: j \in I_1\}$, there exist constants $c_{j,n}$ such that
\[
w_{1,n} = \sum_{j \in I_1} c_{j,n} \delta_{j,n} P_n \psi_{j,n}.
\]
Now we consider the scalar product in $H_0^1(\Omega_\eps)$ of both sides in \eqref{eq phi} with $\delta_{i,n}  P_n \psi_{i,n}$, with $i \in I_1$: as $h_{1,n}, \phi_{1,n} \in K_{1,n}^\perp$, we obtain
\begin{equation}\label{eq phi test} 
\begin{split}
\d_{i,n} \int_{\Omega_n} \nabla w_{1,n} & \cdot \nabla (P_n \psi_{i,n})
= 3\delta_{i,n} \int_{\Omega_n} ( \sum_{j \in I_1} P_n U_{j,n})^{2} \phi_{1,n} (P_n \psi_{i,n}) \\
&+   \delta_{i,n} \beta \int_{\Omega_n}  ( \sum_{j \in I_2} P_n U_{j,n})^{2} \phi_{1,n} (P_n \psi_{i,n}) \\
&+ 2 \beta \delta_{i,n} \int_{\Omega_n}    ( \sum_{j \in I_1} P_n U_{j,n})   ( \sum_{j \in I_2} P_n U_{j,n})\phi_{2,n} (P_n \psi_{i,n}).
\end{split}
\end{equation}
The left hand side can be estimated using \cite[Remark 5.2]{GeMuPi} and \eqref{pa U con U} (see also \cite[p. 417]{PistoiaTavares}, noting that therein $\psi_{i,n}$ corresponds to $\delta_{i,n}\psi_{i,n}$ in the present paper) and obtaining 
\begin{equation}\label{eq phi test1}
\begin{split}
 \int_{\Omega_n} \nabla w_{1,n}  \cdot \nabla (\d_{i,n} P_n \psi_{i,n})  = c_{i,n} (\sigma_0 + o(1)) + o(1) \sum_{\substack{j \in I_1 \\ j \neq i}} c_{j,n}
\end{split}
\end{equation}
as $n \to \infty$, where
\[
\sigma_0 = 3 \alpha_4^4 \int_{\R^4} \frac{( |y|^2 -1)^2}{(1+ |y|^2)^6}\,dy.
\]
The first integral on the right hand side in \eqref{eq phi test} can be estimated as in \cite[Formula (5.7)]{GeMuPi}:
\begin{equation}\label{eq phi test2}
3\delta_{i,n} \int_{\Omega_n} ( \sum_{j \in I_1} P_n U_{j,n})^{2} \phi_{1,n} (P_n \psi_{i,n}) = o (1)
\end{equation}
as $n \to \infty$. 
We have now to estimate the interaction terms. To this purpose, we observe that by H\"older and Sobolev inequality, and by \eqref{pa U con U},
\begin{equation}
\label{eq phi test3.2}
\begin{split}
\left| \int_{\Omega_n}  ( \sum_{j \in I_2} P_n U_{j,n})^{2} \phi_{1,n} (P_n \psi_{i,n})\right| & \le \left( \int_{\Omega_n} ( \sum_{j \in I_2} P_n U_{j,n})^{\frac83} |P_n \psi_{i,n}|^\frac43 \right)^\frac34 \|\phi_{1,n}\|_{L^{4}} \\
& \leq C \left( \int_{\Omega_n} ( \sum_{j \in I_2} U_{j,n})^{\frac83} |\psi_{i,n}|^\frac43 \right)^\frac34 \|\phi_{1,n}\|_{H_0^1} + h.o.t. \\
& \le \frac{C}{\delta_{i,n}} \sum_{j \in I_2}\left( \int_{\Omega_n} U_{j,n}^{\frac83} U_{i,n}^\frac43 \right)^\frac34 + h.o.t. 
\end{split}
\end{equation}
as $n \to \infty$. The precise rate of the higher order terms ($h.o.t.$) does not play any role, and in any case can be derived using Lemmas \ref{lemma:estimates} and \ref{lemma:estimates2}. Moreover, the leading integral on the right hand side can be estimated using Lemma \ref{lem: interaction 1}, obtaining
\[
\int_{\Omega_n} U_{j,n}^{\frac83} U_{i,n}^\frac43 = \begin{cases} O \left( \left( \frac{\delta_{i,n}}{\delta_{j,n}}\right)^\frac43\right) & \text{if $i>j$} \\ O \left( \left( \frac{\delta_{j,n}}{\delta_{i,n}}\right)^\frac43\right) & \text{if $j>i$}.
\end{cases}
\]
Coming back to \eqref{eq phi test3.2}, we have
\begin{equation}\label{eq phi test4}
\begin{split}
\d_{i,n} \left| \int_{\Omega_n}  ( \sum_{j \in I_2} P_n U_{j,n})^{2} \phi_{1,n} (P_n \psi_{i,n})\right| & = \begin{cases}  O \left( \frac{\delta_{i,n}}{\delta_{j,n}}\right) = o(1) & \text{if $i>j$} \\
 O \left( \frac{\delta_{j,n}}{\delta_{i,n}}\right)= o(1) & \text{if $i<j$}
\end{cases}
\end{split}
\end{equation}
as $n \to \infty$, which proves that the second integral on the right hand side in \eqref{eq phi test} is of order $o(\d_{i,n})$. As far as the third integral is concerned, we note that
\begin{equation}\label{eq phi test3}
\begin{split}
\Bigg| \int_{\Omega_n}    ( \sum_{j \in I_1} P_n U_{j,n})   & ( \sum_{j \in I_2} P_n U_{j,n})\phi_{2,n} (P_n \psi_{i,n}) \Bigg| \\
&  \le \left( \int_{\Omega_n} ( \sum_{j \in I_1} P_n U_{j,n})^{\frac43}  ( \sum_{j \in I_2} P_n U_{j,n})^\frac43 |P_n \psi_{i,n}|^\frac43 \right)^\frac34 \|\phi_{2,n}\|_{L^{4}} \\
& \le C \left( \int_{\Omega_n} ( \sum_{j \in I_1} U_{j,n})^{\frac43}  ( \sum_{j \in I_2} U_{j,n})^\frac43 |\psi_{i,n}|^\frac43 \right)^\frac34 \|\phi_{2,n}\|_{H_0^1} + h.o.t. \\
& \le \frac{C}{\delta_{i,n}} \sum_{h \in I_1} \sum_{j \in I_2}\left( \int_{\Omega_n} U_{h,n}^{\frac43} U_{j,n}^\frac43 U_{i,n}^\frac43 \right)^\frac34  + h.o.t. \\
& = \frac{C}{\delta_{i,n}} o(1)
\end{split}
\end{equation}
as $n \to \infty$. The last inequality follows by Lemma \ref{lem: int a 3} if $h \neq i$, and by Lemma \ref{lem: interaction 1} if $h=i$. In any case
\[
\delta_{i,n} \left|\int_{\Omega_n}    ( \sum_{j \in I_1} P_n U_{j,n})   ( \sum_{j \in I_2} P_n U_{j,n})\phi_{2,n} (P_n \psi_{i,n})\right| = o(1)
\]
as $n \to \infty$. To sum up, by expanding \eqref{eq phi test}, we proved that for every index $i \in I_1$ it results that
\begin{equation}\label{791}
c_{i,n} (\sigma_0 + o(1)) + o(1) \sum_{\substack{j \in I_1 \\ j \neq i}} c_{j,n} = o(1)
\end{equation}
as $n \to \infty$. From this and by Cramer's rule, we deduce that $c_{ i,n}\to 0$ for every $i\in I_1$. From this, the conclusion $\|w_{1,n}\| \to 0$ follows.
\end{proof}

Let us set now $z_{i,n}:= \phi_{i,n} - h_{i,n} - w_{i,n}$. Notice that, since $\|h_{i,n}\|_{H^1_0(\Omega_n)}, \|w_{i,n}\|_{H^1_0(\Omega_n)} \to 0$, we have $\|z_{1,n}\|_{H^1_0(\Omega_n)}^2 + \|z_{2,n}\|_{H^1_0(\Omega_n)}^2 \to 1$. In terms of $z_{i,n}$, equation \eqref{eq phi} can be rewritten as
\begin{equation}\label{eq zi}
\begin{split}
z_{1,n}    =  \Ical^*\Bigg\{ & \Big[3(\sum_{j \in I_1}P_n U_{j,n})^{2}    + \beta (\sum_{j \in I_2} P_n U_{j,n})^2\Big]  (z_{1,n} +h_{1,n} + w_{1,n})  \\
&+2\beta (\sum_{j \in I_1} P_n U_{j,n}) (\sum_{j \in I_2} P_n U_{j,n})  (z_{2,n} + h_{2,n} + w_{2,n})  \Bigg\}. \end{split}
\end{equation}
Of course, a similar equation holds for $z_{2,n}$.

\begin{lemma}\label{lem: step 2}
It results that at least one of the following lower estimates holds:
\begin{multline*}
\liminf_{n \to \infty} \left\{ \int_{\Omega_n} \bigg[3(\sum_{j \in I_1}P_n U_{j,n})^{2}    + \beta (\sum_{j \in I_2} P_n U_{j,n})^2\bigg] z_{1,n}^2 \right. \\
\left. + 2 \beta
\int_{\Omega_n} (\sum_{j \in I_1} P_n U_{j,n}) (\sum_{j \in I_2} P_n U_{j,n}) z_{1,n}  z_{2,n} \right\} >0,
\end{multline*}
or
\begin{multline*}
\liminf_{n \to \infty} \left\{ \int_{\Omega_n} \bigg[3(\sum_{j \in I_2}P_n U_{j,n})^{2}    + \beta (\sum_{j \in I_1} P_n U_{j,n})^2\bigg] z_{2,n}^2 \right. \\
\left. + 2 \beta
\int_{\Omega_n} (\sum_{j \in I_1} P_n U_{j,n}) (\sum_{j \in I_2} P_n U_{j,n}) z_{1,n}  z_{2,n} \right\} >0.
\end{multline*}
\end{lemma}

\begin{proof}
Since $\|z_{1,n}\|_{H^1_0(\Omega_n)}^2 + \|z_{2,n}\|_{H^1_0(\Omega_n)}^2 \to 1$, we can suppose that up to a subsequence $\{\|z_{1,n}\|_{H^1_0(\Omega_n)}^2\}_n$ or $\{\|z_{2,n}\|_{H^1_0(\Omega_n)}^2\}_n$ is uniformly bounded from below by $1/2$. Suppose for instance that $\{\|z_{1,n}\|^2_{H^1_0(\Omega_n)}\}_n$ is bounded from below. Then we test equation \eqref{eq zi} with $z_{1,n}$, obtaining
\begin{equation}\label{eq zi 2}
\begin{split}
\|z_{1,n}\|_{H^1_0(\Omega_n)}^2  &  =  \int_{\Omega_n}\bigg[3(\sum_{j \in I_1}P_n U_{j,n})^{2}    + \beta (\sum_{j \in I_2} P_n U_{j,n})^2\bigg] z_{1,n}^2 \\
& + 2 \beta
\int_{\Omega_n} (\sum_{j \in I_1} P_n U_{j,n}) (\sum_{j \in I_2} P_n U_{j,n}) z_{1,n}  z_{2,n} \\
& + 3 \int_{\Omega_n}(\sum_{j \in I_1}P_n U_{j,n})^{2} (h_{1,n} + w_{1,n})z_{1,n} + \beta (\sum_{j \in I_2} P_n U_{j,n})^2 (h_{1,n} + w_{1,n}) z_{1,n} \\
&+2\beta \int_{\Omega_n} (\sum_{j \in I_1} P_n U_{j,n}) (\sum_{j \in I_2} P_n U_{j,n})  (h_{2,n} + w_{2,n}) z_{1,n}.
\end{split}
\end{equation}
Arguing as in \cite[Formula (5.12)]{GeMuPi}, we can easily check that the last two integrals are $0$. Therefore, in this case the first $\liminf$ in the thesis is positive. If $\{\|z_{2,n}\|_{H^1_0(\Omega_n)}^2\}_n$ is bounded from below, in the same way we find that the second $\liminf$ is positive.
\end{proof}

We aim to obtain a contradiction with Lemma \ref{lem: step 2}. To this end, we fix $\rho>0$ so that $B_\rho \subset \subset \Omega$, and we decompose $B_\rho \setminus \overline{B_{\eps_n}}$ into the union of disjoint annuli as follows:
\[
B_\rho \setminus \overline{B_{\eps_n}} = \bigcup_{\ell=1}^k \cA_{\ell,n}, \quad \text{where} \quad \cA_{\ell,n} = B_{\sqrt{\delta_{\ell,n} \delta_{\ell-1,n}}}  \setminus \overline{B_{\sqrt{\delta_{\ell,n} \delta_{\ell+1,n}}}} \quad \text{for $\ell=1,\dots,k$},
\]
with the convention $\delta_{0,n} = \delta_{1,n}^{-1} \rho^2$ and $\delta_{k+1,n} = \delta_{k,n}^{-1} \eps_n^2$. Recall from Remark \ref{rem:ratesrelations} that $\delta_{l+1,n}/\delta_{l,n}\to 0$ as $n\to \infty$. We also set
\[
\cB_{\ell,n} = B_{2\sqrt{\delta_{\ell,n} \delta_{\ell-1,n}}}  \setminus \overline{B_{\sqrt{\delta_{\ell,n} \delta_{\ell+1,n}}/2}},
\]
and, for every $\ell=1,\dots,k$, we define a cut-off function $\chi_{\ell,n} \in C^\infty_c(\R^N)$ with the properties that
\begin{equation}\label{chi prop}
\begin{cases}
\chi_{\ell,n} = 1 \quad \text{in $\cA_{\ell,n}$}, \quad \chi_{\ell,n} = 0 \quad \text{in $\R^4\setminus \cB_{\ell,n}$}, \\
 |\nabla \chi_{\ell,n}| \le \frac{C}{\sqrt{\delta_{\ell,n} \delta_{\ell-1,n }}}, \quad |D^2 \chi_{\ell,n}| \le \frac{C}{\delta_{\ell,n} \delta_{\ell-1,n }} \quad \text{in $\R^4$},
\end{cases}
\end{equation}
for a positive universal constant $C$. Finally, we define for $\ell=1,\dots,k$ and $i=1,2$ the $\cD^{1,2}(\R^4)$ functions
\[
\hat z_{i,n}^\ell(x):= \delta_{\ell,n} z_{i,n}( \delta_{\ell,n} x) \chi_{\ell,n}(\delta_{\ell,n} x) \quad \text{for }x \in \frac{\cB_{\ell,n}}{\delta_{\ell,n}}=: \hat \cB_{\ell,n}, 
\]
naturally extended by $0$ in $\R^4 \setminus \hat \cB_{\ell,n}$. We have $\hat z_{i,n}^\ell(x)= \delta_{\ell,n} z_{i,n}( \delta_{\ell,n} x)$ if $x \in \tilde \cA_{\ell,n}:= \cA_{\ell,n}/\delta_{\ell,n}$.

\begin{lemma}\label{lem: step 3}
It results that $\hat z_{i,n}^\ell \to 0$ weakly in $\cD^{1,2}(\R^4)$, and strongly in $L^q_{\loc}(\R^4)$, for every $q \in [2,2^*)$, for every $i  = 1,2$, $\ell =1,\dots,k$. 
\end{lemma}

\begin{proof}
We have
\[
\nabla \hat z_{i,n}^\ell(x) = \delta_{\ell,n}^2 \left[ \chi_{\ell,n}(\delta_{\ell,n} x) \nabla z_{i,n}(\delta_{\ell,n} x) +  z_{i,n}(\delta_{\ell,n} x) \nabla \chi_{\ell,n}(\delta_{\ell,n} x) \right]
\]
and
\begin{equation}\label{delta z}
\begin{split}
\Delta \hat z_{i,n}^\ell(x) = \delta_{\ell,n}^3 \big[\chi_{\ell,n}(\delta_{\ell,n} x) \Delta z_{i,n}(\delta_{\ell,n} x) & + 2  \nabla z_{i,n}(\delta_{\ell,n} x) \cdot \nabla \chi_{\ell,n}(\delta_{\ell,n} x)  \\
&  +  z_{i,n}(\delta_{\ell,n} x) \Delta \chi_{\ell,n}(\delta_{\ell,n} x) \big]
\end{split}
\end{equation}
for $x \in \tilde \cB_{\ell,n}$, that is,
\[
\frac{1}{2} \sqrt{ \frac{\d_{\ell+1,n}}{\d_{\ell,n}}} < |x| < 2 \sqrt{ \frac{\d_{\ell-1,n}}{\d_{\ell,n}}}.
\]
Notice that $\tilde \cB_{\ell,n}$ exhausts $\R^N$ as $n \to \infty$, by Remark \ref{rem:ratesrelations}. Now
\begin{align*}
\int_{\R^4} |\nabla \hat z_{i,n}^\ell|^2 & \le 2\delta_{\ell,n}^4  \int_{\tilde \cB_{\ell,n}} \left(|\nabla z_{i,n}(\delta_{\ell,n} x)|^2 + z_{i,n}^2(\delta_{\ell,n} x) |\nabla \chi_{\ell,n}(\delta_{\ell,n}x)|^2\right) \,dx \\
& =  2 \int_{\cB_{\ell,n}} \left(|\nabla z_{i,n}(y)|^2 + z_{i,n}^2(y) |\nabla \chi_{\ell,n}(y)|^2\right) \,dy.
\end{align*}
The integral of $|\nabla z_{i,n}|^2$ is clearly bounded, since $\|z_{i,n}\|_{H^1_0(\Omega_n)} \le 1$. Also, by \eqref{chi prop},
\begin{align*}
\int_{\cB_{\ell,n}} z_{i,n}^2 |\nabla \chi_{\ell,n}|^2  & \le \frac{C}{\delta_{\ell,n} \delta_{\ell-1,n}}  \int_{\cB_{\ell,n}} z_{i,n}^2 \le  \frac{C}{\delta_{\ell,n} \delta_{\ell-1,n}}  |\cB_{\ell,n}|^\frac12 \|z_{i,n}\|_{H^1_0(\Omega_n)}^2 \\
						&\le C \delta_{\ell,n} \delta_{\ell-1,n}  \|z_{i,n}\|_{H^1_0(\Omega_n)}^2 \le C,
\end{align*}
and we infer that $\|\hat z_{i,n}^\ell\|_{\cD^{1,2}(\R^4)} \le C$. Then, up to a subsequence, we have that $\hat z_{i,n}^\ell \rightharpoonup \hat z_i^\ell$ weakly in $\cD^{1,2}$, and $\hat z_{i,n}^\ell \to \hat z_i^\ell$ strongly in $L^q_{\loc}(\R^4)$ for $q \in [2,2^*)$. The equation satisfied by the weak limit can be determined using \eqref{eq zi} and \eqref{delta z}: for every $\varphi \in C^\infty_c(\R^4 \setminus \{0\})$, by combining \eqref{eq zi} with \eqref{delta z} we have that
\begin{align*}
\int_{\R^N} &\nabla \hat z_{1,n}^\ell \cdot \nabla \varphi \\
& = \delta_{\ell,n}^3 \int_{\tilde \cB_{\ell,n}} \chi_{\ell,n}(\delta_{\ell,n} x) \left[ 3 \left( \sum_{i \in I_1} P_n U_{j,n}(\delta_{\ell,n} x)\right)^2+ \beta \left( \sum_{j \in I_2} P_n U_{j,n}(\delta_{\ell,n} x)\right)^2\right]   \\
& \hphantom{\delta_{\ell,n}^3 \int_{\tilde \cB_{\ell,n}} \chi_{\ell,n}(\delta_{\ell,n} x) \qquad} \cdot \left( z_{1,n}(\delta_{\ell,n} x) + h_{1,n}(\delta_{\ell,n} x) + w_{1,n}(\delta_{\ell,n} x)\right) \varphi(x) \,dx \\
& + 2\beta \delta_{\ell,n}^3 \int_{\tilde \cB_{\ell,n}} \chi_{\ell,n}(\delta_{\ell,n} x) \left( \sum_{i \in I_1} P_n U_{j,n}(\delta_{\ell,n} x)\right)\left( \sum_{i \in I_2} P_n U_{j,n}(\delta_{\ell,n} x)\right)\\
& \hphantom{ + 2\beta \delta_{\ell,n}^3 \int_{\tilde \cB_{\ell,n}} \chi_{\ell,n}(\delta_{\ell,n} x) \qquad} \cdot  (z_{2,n}(\delta_{\ell,n} x) + h_{2,n}(\delta_{\ell,n} x) + w_{2,n}(\delta_{\ell,n} x)) \varphi(x)\,dx \\
& -  \delta_{\ell,n}^3 \int_{\tilde \cB_{\ell,n}} \left( 2\nabla \chi_{\ell,n}(\delta_{\ell,n} x) \cdot \nabla z_{1,n}(\delta_{\ell,n} x) + z_{1,n}(\delta_{\ell,n} x) \Delta \chi_{\ell,n}(\d_{\ell,n} x)\right) \varphi(x)\,dx.
\end{align*}
The last integral and all the terms involving $h_{i,n}$ and $w_{i,n}$ tend to $0$ as $n \to \infty$, exactly as in \cite[Formula (5.20)]{GeMuPi}. Therefore,
\begin{align*}
&\int_{\R^N} \nabla \hat z_{1,n}^\ell \cdot \nabla \varphi = o(1)  \\
& + \delta_{\ell,n}^3 \int_{\tilde \cB_{\ell,n}} \chi_{\ell,n}(\delta_{\ell,n} x) \left[ 3 \left( \sum_{i \in I_1} P_n U_{j,n}(\delta_{\ell,n} x)\right)^2+ \beta \left( \sum_{j \in I_2} P_n U_{j,n}(\delta_{\ell,n} x)\right)^2\right]   z_{1,n}(\delta_{\ell,n} x)  \varphi(x) \,dx \\
& + 2\beta \delta_{\ell,n}^3 \int_{\tilde \cB_{\ell,n}} \chi_{\ell,n}(\delta_{\ell,n} x) \left( \sum_{i \in I_1} P_n U_{j,n}(\delta_{\ell,n} x)\right)\left( \sum_{i \in I_2} P_n U_{j,n}(\delta_{\ell,n} x)\right)z_{2,n}(\delta_{\ell,n} x)  \varphi(x)\,dx \\
& = o(1)  + \delta_{\ell,n}^2 \int_{\tilde \cB_{\ell,n}} 3 \left( \sum_{i \in I_1} P_n U_{j,n}(\delta_{\ell,n} x)\right)^2 \hat z_{1,n}^\ell(x)  \varphi(x)\,dx \\
& + \beta \delta_{\ell,n}^2 \int_{\tilde \cB_{\ell,n}}  \left( \sum_{j \in I_2} P_n U_{j,n}(\delta_{\ell,n} x)\right)^2  \hat z_{1,n}(x)  \varphi(x) \,dx \\
& + 2\beta \delta_{\ell,n}^2 \int_{\tilde \cB_{\ell,n}}  \left( \sum_{i \in I_1} P_n U_{j,n}(\delta_{\ell,n} x)\right)\left( \sum_{i \in I_2} P_n U_{j,n}(\delta_{\ell,n} x)\right)\hat z_{2,n}^\ell(x)  \varphi(x)\,dx \\
& =: o(1) + A_1 + A_2 + A_3.
\end{align*}
In order to study the behavior of the integrals as $n \to \infty$, it is convenient to observe (see Lemma \ref{lemma:estimates}) that, if $\ell \in I_1$, then
\begin{equation}\label{26101}
\begin{split}
\sum_{j \in I_1} P_n U_{j,n} (\delta_{\ell,n} x) & = \sum_{j \in I_1} U_{j,n}(\delta_{\ell,n} x) + h.o.t.  = \sum_{j \in I_1} \alpha_4 \frac{\delta_{j,n}}{\delta_{j,n}^2 + \delta_{\ell,n}^2 |x|^2} + h.o.t.\\
& = \frac{1}{\delta_{\ell,n}} U_{1,0}(x) + \sum_{\substack{j \in I_1 \\ j \neq \ell}} \alpha_4 \frac{\delta_{j,n}}{\delta_{j,n}^2 + \delta_{\ell,n}^2 |x|^2} + h.o.t. \\
& = \frac{1}{\delta_{\ell,n}} U_{1,0}(x) + \sum_{\substack{j \in I_1 \\ j < \ell}} O\left(\frac{1}{\delta_{j,n}}\right) + \sum_{\substack{j \in I_1 \\ j > \ell}} O\left(\frac{\delta_{j,n}}{\delta_{\ell,n}^2 |x|^2}\right) + h.o.t.,
\end{split}
\end{equation}
as $n \to \infty$. If instead $\ell\not \in I_1$, then we have a similar expansion, but without the term $U_{1,0}(x)/\delta_{\ell,n}$. We focus at first on the first possibility. We have, 
\begin{equation}\label{19101}
\begin{split}
A_1 &= 3\d_{\ell,n}^2 \int_{\tilde \cB_{\ell,n}} \left( \frac{U_{1,0}(x)}{\delta_{\ell,n}}\right)^2 \hat z_{1,n}^{\ell}(x) \varphi(x) + \left[ \sum_{\substack{j \in I_1 \\ j<\ell}}  O\left( \frac{1}{\d_{j,n}^2}\right) +  \sum_{\substack{j \in I_1 \\ j > \ell}} O\left( \frac{\d_{j,n}^2}{\d_{\ell,n}^4 |x|^4}\right)\right]  \hat z_{1,n}^{\ell}(x) \varphi(x)\, dx \\
&  + 3\d_{\ell,n}^2 \int_{\tilde \cB_{\ell,n}}  \frac{2}{\delta_{\ell,n}} U_{1,0}(x) \left[ \sum_{\substack{j \in I_1 \\ j<\ell}} O\left( \frac{1}{\d_{j,n}}\right) + \sum_{\substack{j \in I_1 \\ j > \ell}} O\left( \frac{\d_{j,n}}{\d_{\ell,n}^2 |x|^2}\right) \right]  \hat z_{1,n}^{\ell}(x) \varphi(x)\, dx \\
& + 3\d_{\ell,n}^2 \int_{\tilde \cB_{\ell,n}} 2 \sum_{\substack{i \in I_1 \\ i<\ell}} \sum_{\substack{j \in I_1 \\ j > \ell}} O \left(\frac{\d_{j,n}}{\d_{i,n} \delta_{\ell,n}^2 |x|^2}\right)  \hat z_{1,n}^{\ell}(x) \varphi(x)\, dx + h.o.t.
\end{split}
\end{equation} 
Now, for every $j<\ell$
\[
\frac{\d_{\ell,n}^2}{\d_{j,n}^2} \left| \int_{\tilde \cB_{\ell,n}} \hat z_{1,n}^{\ell} \varphi \right| \le o(1) \|\hat z_{1,n}^{\ell}\|_{L^4} \to 0,
\]
and, by Lemma \ref{lem: interaction 0},
\begin{align*}
\left| \frac{\d_{\ell,n}}{\d_{j,n}} \int_{\tilde \cB_{\ell,n}} U_{1,0} \hat z_{1,n}^\ell \varphi \right| & \le \frac{\d_{\ell,n}}{\d_{j,n}} \|U_{1,0}\|_{L^4} \|\hat z_{1,n}^\ell\|_{L^4} \|\varphi\|_{L^2} \to 0.
\end{align*}
Moreover, for every $j>\ell$, using the fact that $\supp \varphi \subset B_R \setminus B_\rho$ for suitable $0<\rho<R$, we have that
\begin{align*}
\frac{\d_{j,n}^2}{\d_{\ell,n}^2} \left| \int_{\tilde \cB_{\ell,n}} \hat z_{1,n}^{\ell}(x) \frac{\varphi(x)}{|x|^4}\, dx \right| &\le  C \frac{\d_{j,n}^2}{\d_{\ell,n}^2} \|\hat z_{1,n}^{\ell}\|_{L^4} \left( \int_{B_R \setminus B_{\rho}} \frac{dx}{|x|^\frac{16}{3}}\right)^\frac34 \le  C \frac{\d_{j,n}^2}{\d_{\ell,n}^2}  \to 0,
\end{align*}
and that
\begin{align*}
\frac{\d_{j,n}}{\d_{\ell,n}} \left| \int_{\tilde \cB_{\ell,n}} U_{1,0}(x) \hat z_{1,n}^{\ell}(x) \frac{\varphi(x)}{|x|^2} \, dx\right| &\le C\frac{\d_{j,n}}{\d_{\ell,n}} \| \hat z_{1,n}^{\ell}\|_{L^4} \left(\int_{B_R \setminus B_{\rho}} \frac{U_{1,0}^\frac43(x)}{|x|^\frac83}\, dx\right)^\frac34 \le C \frac{\d_{j,n}}{\d_{\ell,n}} \to 0.
\end{align*}
The previous estimates yield
\begin{equation}\label{19102}
A_1 = 3\int_{\tilde \cB_{\ell,n}}  U_{1,0}^2 \hat z_{1,n}^{\ell} \varphi + o(1) \to 3 \int_{\R^4} U_{1,0}^2 \hat z_1^{\ell} \varphi
\end{equation}
as $n \to \infty$, for every $\varphi \in C^\infty_c(\R^4 \setminus \{0\})$.

Notice that, in the above computations, we never used the fact that the indexes $j$ were in $I_1$. Therefore, we directly deduce that
\begin{equation}\label{19103}
A_2 \to 0 \qquad \text{as $n \to \infty$}.
\end{equation}
Finally, in an analogue way
\begin{equation}\label{19104}
\begin{split}
A_3 &= 2 \beta \d_{\ell,n}^2 \int_{\tilde \cB_{\ell,n}} \left[ \frac{1}{\delta_{\ell,n}} U_{1,0}(x) + \sum_{\substack{j \in I_1 \\ j < \ell}} O\left(\frac{1}{\delta_{j,n}}\right) + \sum_{\substack{j \in I_1 \\ j > \ell}} O\left(\frac{\delta_{j,n}}{\delta_{\ell,n}^2 |x|^2}\right)\right]  \cdot \\
& \hphantom{2 \beta \d_{\ell,n}^2 \int_{\tilde \cB_{\ell,n}} \qquad \quad }    \cdot \left[ \sum_{\substack{j \in I_2 \\ j < \ell}} O\left(\frac{1}{\delta_{j,n}}\right) + \sum_{\substack{j \in I_2 \\ j > \ell}} O\left(\frac{\delta_{j,n}}{\delta_{\ell,n}^2 |x|^2}\right) \right] \hat z_{1,n}^\ell(x) \varphi(x)\, dx+ h.o.t. \to 0
\end{split}
\end{equation}
as $n \to \infty$. 

Collecting together \eqref{19102}, \eqref{19103} and \eqref{19104}, and coming back to \eqref{19101}, we finally obtain that the weak limit of $\hat z_{1,n}^\ell$ satisfies
\[
-\Delta \hat z_1^\ell = 3 U_{1,0}^2 \hat z_1^\ell \qquad \text{in $\R^4 \setminus \{0\}$}.
\]
Let now $\theta_\rho \in C^\infty(\R^N)$ be such that $\theta_\rho \equiv 1$ in $B_{2\rho}^c$, $\theta_\rho \equiv 0$ in $B_\rho$, and $|\nabla \theta_\rho| \le C/\rho$; and let $\varphi \in C^\infty_c(\R^4)$; testing the above equation with $\theta_\rho \varphi$, and passing to the limit as $\rho \to 0^+$, using the fact that $\hat z_1^\ell \in \cD^{1,2}(\R^4)$ (since it is the weak limit of $\cD^{1,2}$ functions), we easily deduce that 
\[
-\Delta \hat z_1^\ell = 3 U_{1,0}^2 \hat z_1^\ell \qquad \text{in the whole space $\R^4$}.
\]
In order to show that $\hat z_1^{\ell} \equiv 0$, recalling that it is symmetric with respect to $0$, it is sufficient to verify that $\hat z_1^\ell \perp \psi_{1,0}$. This can be done exactly as in \cite[Formula (5.19)]{GeMuPi}, and completes the proof.

It still remains to analyze the case $\ell \not \in I_1$. In such a situation we can proceed exactly as before, but this time we end up with 
\[
-\Delta \hat z_1^\ell = \beta U_{1,0}^2 \hat z_1^\ell \qquad \text{in the whole space $\R^4$}.
\]
Since $\beta<0$ and $\hat z_1^\ell \in \cD^{1,2}(\R^4)$, we infer that
\[
0 \le \int_{\R^N} |\nabla \hat z_1^\ell|^2 = \beta \int_{\R^N} (U_{1,0} \hat z_1^\ell)^2 \le 0,
\] 
and the conclusion follows also in this case. 
\end{proof}

\begin{proof}[Conclusion of the proof of Proposition \ref{lem: linear part}]
Using Lemma \ref{lem: step 3}, we will obtain a contradiction with Lemma \ref{lem: step 2}. Let us consider
\begin{equation}\label{26102}
\int_{\Omega_n} (\sum_{j \in I_1}P_n U_{j,n})^{2} z_{1,n}^2 \le C \sum_{i \in I_1} \int_{\Omega_n} U_{i,n}^2 z_{1,n}^2 = C  \sum_{i \in I_1} \left(\int_{\Omega_n \setminus B_\rho}  U_{i,n}^2 z_{1,n}^2+ \sum_{\ell=1}^k \int_{\cA_{\ell,n}}  U_{i,n}^2 z_{1,n}^2\right).
\end{equation}
We show that the right hand side tends to $0$ as $n \to \infty$. At first, we have 
\begin{equation}\label{26103}
\int_{\Omega_n \setminus B_\rho}  U_{i,n}^2 z_{1,n}^2 \le C \d_{i,n}^2 \int_{\Omega_n \setminus B_\rho} z_{1,n}^2 \le C \d_{i,n}^2 \|z_{i,n}\|_{H^1_0(\Omega_n)}^2 \to 0.
\end{equation}
Now, let $i \neq \ell$. Then we have
\[
\int_{\cA_{\ell,n}}  U_{i,n}^2 z_{1,n}^2 \le C \|z_{i,n}\|_{H^1_0(\Omega_n)}^2 \left( \int_{\mathcal{A}_{\ell,n}} U_{i,n}^4 \right)^\frac12 = C \|z_{i,n}\|_{H^1_0(\Omega_n)}^2 \left( \int_{\mathcal{A}_{\ell,n}/\delta_{i,n}} U_{1,0}^4 \right)^\frac12.
\]
Since $i \neq \ell$, we have that 
\[
\frac{\mathcal{A}_{\ell,n}}{\delta_{i,n}} \subset \begin{cases} B\left(0, \frac{\sqrt{\d_{\ell-1,n} \d_{\ell,n}}}{\d_{i,n}}\right) & \text{if $i<\ell$} \\ 
\R^N \setminus B\left(0, \frac{\sqrt{\d_{\ell+1,n} \d_{\ell,n}}}{\d_{i,n}}\right) & \text{if $i>\ell$},
\end{cases}
\]
and 
\[
\frac{\sqrt{\d_{\ell-1,n} \d_{\ell,n}}}{\d_{i,n}}\to 0 \quad \text{if $i < \ell$, and} \quad \frac{\sqrt{\d_{\ell+1,n} \d_{\ell,n}}}{\d_{i,n}} \to +\infty \quad \text{if $i>\ell$}.
\]
Therefore, the fact that 
\begin{equation}\label{261031}
\int_{\cA_{\ell,n}}  U_{i,n}^2 z_{1,n}^2 \to 0 \quad \text{as $n \to \infty$}, \ \forall i \neq \ell
\end{equation}
follows from the integrability of $U_{1,0}^4$ on $\R^N$. If moreover $i = \ell$, recalling that $\tilde \cA_{\ell,n} = \cA_{\ell,n}/\delta_{\ell,n}$ we have
\begin{equation}\label{26104}
\int_{\cA_{\ell,n}}  U_{\ell,n}^2 z_{1,n}^2 = \d_{\ell,n}^2  \int_{\tilde \cA_{\ell,n}} U_{l,n}^2(\d_{\ell,n} x) (\hat z_{1,n}^\ell)^2(x)\,dx =  \alpha_4^2 \int_{\R^N} \left(\frac{1}{1+|x|^2}\right)^2 (\hat z_{1,n}^{\ell})^2(x) \, dx + o(1) \to 0
\end{equation}
as $n \to \infty$, since $U_{1,0}^2 \in L^2(\R^N)$ and $(\hat z_{1,n}^{\ell})^2  \rightharpoonup 0$ weakly in $L^{2}(\R^4)$ by Lemma \ref{lem: step 3}. By \eqref{26103} and \eqref{261031}, we obtain in \eqref{26102} that
\begin{equation}\label{26105}
\int_{\Omega_n} (\sum_{j \in I_1}P_n U_{j,n})^{2} z_{1,n}^2 \to 0 \qquad \text{as $n \to \infty$}.
\end{equation}
In a completely analogue way, we also have
\begin{equation}\label{26106}
\int_{\Omega_n} (\sum_{j \in I_2}P_n U_{j,n})^{2} z_{1,n}^2 \to 0 \qquad \text{as $n \to \infty$}.
\end{equation}
Finally, 
\begin{equation}\label{26107}
\Bigg|\int_{\Omega_n} (\sum_{j \in I_1} P_n U_{j,n}) (\sum_{j \in I_2} P_n U_{j,n}) z_{1,n}  z_{2,n}\Bigg| \le \left( \sum_{j \in I_1}\int_{\Omega_n} U_{j,n}^2  z_{1,n}^2\right)^2  \left( \sum_{j \in I_2}\int_{\Omega_n} U_{j,n}^2  z_{2,n}^2\right)^2 \to 0
\end{equation}
as $n \to \infty$. But \eqref{26102}, estimates \eqref{26105}, \eqref{26106} and \eqref{26107} imply that 
\begin{multline*}
\liminf_{n \to \infty} \left\{ \int_{\Omega_n} \bigg[3(\sum_{j \in I_1}P_n U_{j,n})^{2}    + \beta (\sum_{j \in I_2} P_n U_{j,n})^2\bigg] z_{1,n}^2 \right. \\
\left. + 2 \beta
\int_{\Omega_n} (\sum_{j \in I_1} P_n U_{j,n}) (\sum_{j \in I_2} P_n U_{j,n}) z_{1,n}  z_{2,n} \right\}  = 0,
\end{multline*}
in contradiction with Lemma \ref{lem: step 2}.
\end{proof}

\subsection{Estimates on the reminder term}

In this subsection we prove the following

\begin{proposition}\label{prop: stima rem}
Let $\eta>0$. There exists $\eps_0>0$ and $C>0$ such that
\[
\| R^i_{\mf{d},\eps} \| \le C \eps^\frac{1}{k+1} \left(\log\left(\frac{1}{\eps}\right)\right)^{-\frac{1}{k+1}},
\]
for $i=1,2$, for every $\textbf{d} \in X_\eta$ and for every $\eps \in (0,\eps_0)$.
\end{proposition}

\begin{proof}
We focus on $i=1$. By continuity of $\Pi_1^\perp$ and of $\mathcal{I}^*$, there exists $C>0$ such that
\begin{equation}\label{st R}
\begin{split}
\|R^1_{\mf{d},\eps}\| &\le C \Big\| f(\sum_{j \in I_1} P_\eps U_j) - \sum_{j \in I_1} f(U_j) + \beta (\sum_{j \in I_1} P_\eps U_j) (\sum_{j \in I_2} P_\eps U_j)^2 \Big\|_{L^{\frac43}} \\
& \le C \Big\| f(\sum_{j \in I_1} P_\eps U_j) - \sum_{j \in I_1} f(U_j)  \Big\|_{L^{\frac43}} + C \Big\| (\sum_{j \in I_1} P_\eps U_j) (\sum_{j \in I_2} P_\eps U_j)^2 \Big\|_{L^{\frac43}} \\
& =:C( \mathcal{A} + \mathcal{B}).
\end{split}
\end{equation}
We estimate separately $\mathcal{A}$ and $\mathcal{B}$. At first we note that 
\begin{equation}\label{st I1 0}
\begin{split}
\mathcal{A} \le \Big\| f(\sum_{j \in I_1} P_\eps U_j) - \sum_{j \in I_1} f(P_\eps U_j)  \Big\|_{L^{\frac43}} + \Big\| \sum_{j \in I_1} ( f(P_\eps U_j) -f(U_j))  \Big\|_{L^{\frac43}} =: A_1 + A_2.
\end{split}
\end{equation}
Recalling that $f(s) = (s^+)^3$, and using the fact that 
\[
(a_1 + \cdots + a_n)^3 \le \left( a_1^3 + \dots + a_n^3\right) + C_n \sum_{1 \le j \neq h \le n} a_j^2 a_h
\]
for a positive constant $C_n$ depending only on $n$, we obtain
\begin{equation}\label{st I1 1}
\begin{split}
A_1^\frac43 & \le C \int_{\Omega_\eps} \Big|\sum_{\substack{j \neq h \\ j,h \in I_1}} (P_\eps U_j)^2 (P_\eps U_h) \Big|^\frac43 \le C \sum_{\substack{j \neq h \\ j,h \in I_1}} \int_{\Omega_\eps} (U_j^2 U_h)^\frac43.
\end{split}
\end{equation}
Let us fix $j \neq h$. Then, by Lemma \ref{lem: interaction 1}, 
\begin{equation}\label{st I1 2}
\begin{split}
\int_{\Omega_\eps} (U_j^2 U_h)^\frac43 \le \begin{cases}  C\left( \frac{\d_h}{\d_j}\right)^\frac43  & \text{if $h >j$} \\
 C\left( \frac{\d_j}{\d_h}\right)^\frac43  & \text{if $h <j$}.
\end{cases}
\end{split}
\end{equation}
Recalling \eqref{eq:rates}, we see that if $h>j$ and $\mf{d} \in X_\eta$
\[
\frac{\d_h}{\d_j} \le C\left( \eps^\frac{h-j}{k+1} \left( \log\left(\frac1{\eps}\right)\right)^{-\frac{h-j}{k+1}} \right) \le C \eps^\frac{2}{k+1} \left(\log \frac{1}{\eps}\right)^{-\frac{2}{k+1}},
\]
where $C$ denotes a positive constant depending on $\eta$ (but not on $\mf{d}$) and we used the fact that $h-j \ge 2$ since $j, h \in I_1$ with $j \neq h$. The same estimate holds in case $h<j$. Plugging this into \eqref{st I1 2}, and coming back to \eqref{st I1 1}, we finally conclude that 
\begin{equation}\label{st I1 21}
A_1 \le C \eps^\frac{2}{k+1} \left(\log \frac{1}{\eps}\right)^{-\frac{2}{k+1}}\le C \eps^\frac{1}{k+1} \left(\log \frac{1}{\eps}\right)^{-\frac{1}{k+1}} .
\end{equation}
Regarding $A_2$, we have
\begin{equation}\label{st I1 3}
\begin{split}
A_2 \le \sum_{j \in I_1} \left\| (P_\eps U_j)^3 - U_j^3\right\|_{L^\frac43} \le \sum_{j \in I_1} \left( \left\| (U_j -P_\eps  U_j)^3\right\|_{L^\frac43} + C \left\| U_j^2 (U_j - P_\eps U_j)\right\|_{L^\frac43} \right).
\end{split}
\end{equation}
Using the estimate for $U_i - P_\eps U_i$ contained in Lemma \ref{lemma:estimates}, we deduce that 
\begin{equation}\label{st I1 4}
\begin{split}
\left\| (U_j -P_\eps  U_j)^3\right\|_{L^\frac43}^\frac43 &= \int_{\Omega_\eps} |U_j - P_\eps U_j|^4 \le C \int_{\Omega_\eps} \left( \d_j^4 + \left(\frac{\eps}{\d_j}\right)^8 \frac{\d_j^4}{|x|^8} \right) \\
& \le C \d_j^4 + C \frac{\eps^8}{\d_j^4} \int_{\eps}^R r^{-5}\,dr  \le C \left( \d_j^4 + \left(\frac{\eps}{\d_j}\right)^4\right).
\end{split}
\end{equation}
In a similar way
\begin{equation}\label{st I1 5}
\begin{split}
\big\| U_j^2 &(U_j -P_\eps  U_j)\big\|_{L^\frac43}^\frac43  \le C \int_{\Omega_\eps} \left( \d_j^\frac{4}{3} U_j^\frac{8}3(x) +  \d_j^\frac43\left(\frac{\eps}{\d_j} \frac{U_j(x)}{|x|} \right)^\frac83 \right)\, dx \\
& \le C \d_j^\frac{8}3 \int_{\Omega_\eps/\d_j} \frac{dx}{(1+|x|^2)^\frac83} + C \left( \frac{\eps}{\d_j}\right)^\frac83 \int_{\Omega_\eps/\d_j} \frac{dx}{(1+|x|^2)^\frac83 |x|^\frac83} \le C \left(\d_j^\frac83 +  \left( \frac{\eps}{\d_j}\right)^\frac83\right).
\end{split}
\end{equation}
Plugging \eqref{st I1 4} and \eqref{st I1 5} into \eqref{st I1 3}, and recalling again Remark \ref{rem:ratesrelations}, we obtain
\begin{equation}\label{st I1 6}
\begin{split}
A_2 \le C \sum_{j \in I_1} \left(\d_j^2 +  \left( \frac{\eps}{\d_j}\right)^2\right) \le C \eps^\frac{1}{k+1} \left(\log\left(\frac{1}{\eps}\right)\right)^{-\frac{1}{k+1}}.
\end{split}
\end{equation}
Therefore, \eqref{st I1 0}, \eqref{st I1 21} and \eqref{st I1 6} give
\begin{equation}\label{st I1 7}
\mathcal{A} \le C \eps^\frac{1}{k+1} \left(\log\left(\frac{1}{\eps}\right)\right)^{-\frac{1}{k+1}},
\end{equation}
and it remains to estimate $\mathcal{B}$. By Lemma \ref{lem: interaction 1}
\begin{equation}\label{st B 1}
\begin{split}
\mathcal{B} & \le C \sum_{(j,h) \in I_1 \times I_2} \Big\| (P_\eps U_j) ( P_\eps U_h)^2 \Big\|_{L^{\frac43}}  \le C \sum_{(j,h) \in I_1 \times I_2} \Big\|U_j  U_h^2 \Big\|_{L^{\frac43}} \\
& =C \sum_{(j,h) \in I_1 \times I_2} \left( \int_{\Omega_\eps} U_j^\frac43 U_h^\frac83\right)^\frac34 =  \begin{cases}  O\left(\frac{\d_j}{\d_h}\right) & \text{if $j>h$} \\  O\left(\frac{\d_h}{\d_j}\right) & \text{if $j<h$},\end{cases}
\end{split}
\end{equation}
and for any $h>j$ we have
\[
\frac{\d_h}{\d_j} \le C \eps^\frac{h-j}{k+1} \left( \log\left(\frac1{\eps}\right)\right)^{-\frac{h-j}{k+1}}\leq C\eps^\frac{1}{k+1} \left( \log\left(\frac1{\eps}\right)\right)^{-\frac{1}{k+1}},
\]
The same estimate holds for $\d_j/\d_h$ in case $j>h$. Therefore, gathering \eqref{st R}, \eqref{st I1 7} and \eqref{st B 1}, we obtain the desired result.  
\end{proof}

\subsection {The nonlinear part: end of the proof of Proposition \ref{prop:C^1phi}}

In virtue of Proposition \ref{lem: linear part}, solving
 the equation
\begin{equation*}
\mf{L}_{\mf{d},\eps}(\phib)=\mf{R}_{\mf{d},\eps}+\mf{N}_{\mf{d},\eps}(\phib).
\end{equation*}
reduces to  finding a fixed point of the operator
\begin{equation*}
\mf{T}_{\mf{d},\eps}(\phib):=\(\mf{L}_{\mf{d},\eps}\)^{-1}\(\mf{R}_{\mf{d},\eps}+\mf{N}_{\mf{d},\eps}(\phib)\).
\end{equation*}
in the
ball
$$\mathcal B_\rho:=\left\{\phib \in \mf{K}^\perp_{\mf{d},\eps}\ :\ \|\phib\|_{H^1_0}\le \rho\eps^\frac{1}{k+1} \left(\log\left(\frac{1}{\eps}\right)\right)^{-\frac{1}{k+1}}\right\}$$
for some $\rho>0.$
It is quite standard to show that  $\mf{T}_{\mf{d},\eps}:\mathcal B_\rho\to\mathcal B_\rho$ is a contraction mapping for $\epsilon$ small enough. Indeed,     Proposition \ref{lem: linear part} together with   straightforward computations lead to
$$\|\mf{T}_{\mf{d},\eps}(\phib)\|_{H^1_0} \le C \(\|\mf{R}_{\mf{d},\eps}\|_{H^1_0}+\|\mf{N}_{\mf{d},\eps}(\phib)\|_{H^1_0}\)\le C \(\|\mf{R}_{\mf{d},\eps}\|_{H^1_0}+\|\phib\|_{H^1_0}^2\)
$$
and
$$
\|\mf{T}_{\mf{d},\eps}(\phib_1-\phib_2)\|_{H^1_0}\le C \( \|\mf{N}_{\mf{d},\eps}(\phib_1)-\mf{N}_{\mf{d},\eps}(\phib_2)\|_{H^1_0}\)\le \ell \| \phib_1- \phib_2 \|_{H^1_0}\ \ \hbox{for some}\ \ell\in(0,1). $$
 A standard argument also shows that the map $\mf d\to \phib^{\mf d,\eps}$ is of class $C^1.$

\begin{remark}\label{rem:section3}
Suppose that, instead of dealing with the set of odd and even numbers of $\{1,\ldots,k\}$ in the two equation case, we are dealing with system \eqref{eq:gensystem} with $m$ equations and with a general partition $I_1,\ldots, I_m$ satisfying (1)--(5). Having already splitted the original problem into $2m$ equations (see Remark \ref{rem:section2}), we can repeat the argument used for $m=2$ without substantial changes, using the fact that each set $I_j$ does not contain consecutive integers.
\end{remark}

 \section{Expansion of the reduced energy}\label{sec:reducedenergy}
 Recall that the energy funcional is given by
 \[
J_\eps(u_1,u_2)=\sum_{i=1}^2  \int_{\Omega_\eps}\left( \frac{|\nabla u_i |^2}{2}-F(u_i)\right)  -\frac{\beta}{2}\int_{\Omega_\eps} u_1^2u_2^2.
\]
where $F(s)=(s^+)^4/4$. Recall that we denote $f(s):=F'(s)=(s^+)^3$. For every $\eta>0$ small fixed, we introduce the reduced functional $\widetilde J_\eps:X_\eta\to \R$ as being
\[
\widetilde J_\eps (\mf{d})=J_\eps\left(\sum\limits_{j\in I_1}P_{\eps}U_{\delta_j}+\phi_1^{\mf{d},\eps},\sum\limits_{j\in I_2}P_{\eps}U_{\delta_j}+\phi_2^{\mf{d},\eps}\right)
\]
This is a $C^1$ functional due to Proposition \ref{prop:C^1phi} and since $\delta_i$ depends on $d_i$ via \eqref{eq:rates}. Finding critical point of  $\widetilde J_\eps$ corresponds to find solutions of our original system, as we prove next.

\begin{lemma}\label{lemma:equivalent}Given $\eps\in (0,\eps_0)$ and $\eta>0$ small, let $\mf{d}\in X_\eta$. We have that
\[
\left(\sum\limits_{j\in I_1}P_{\eps}U_{\delta_j,0}+\phi_1^{\mf{d},\eps},\sum\limits_{j\in I_2}P_{\eps}U_{\delta_j,0}+\phi_2^{\mf{d},\eps}\right) \text{ is a solution of } \eqref{eq:systemwithf}
\]
if, and only if,
\[
\mf{d}\text{ is a critical point of } \widetilde J_\eps.
\]
\end{lemma}
\begin{proof} To simplify notations, define $V_i^{\mf{d},\eps}:= \sum_{j\in I_i} P_\eps U_{\delta_j} + \phi_i^{\mf{d},\eps}$ for $i=1,2$. From \eqref{eq:rates} we see that
\begin{equation}\label{eq:reduced_aux1}
\partial_{d_l} \widetilde J_\eps(\mf{d})=\epsilon^{l\over k+1}\(\log{1\over\epsilon}\)^{{1\over2}-{l\over k+1}} J_\eps'(V_1^{\mf{d},\eps},V_2^{\mf{d},\eps})[\partial_{\delta_l} V_1^{\mf{d},\eps},\partial_{\delta_l} V_2^{\mf{d},\eps}]
\end{equation}
Hence, if $(V_1^{\mf{d},\eps},V_2^{\mf{d},\eps})$ solves \eqref{eq:systemwithf} then $J_\eps'(V_1^{\mf{d},\eps},V_2^{\mf{d},\eps})=0$ and so $\tilde J_\eps'(\mf{d})=0$. Conversely, assume $\mf{d}\in X_\eta$  is a solution of $\widetilde J_\eps'(\mf{d})=0$.  For $i\in \{1,2\}$ and $l\in I_i$, recalling that $\psi_l:=\partial_{\delta_l} U_{\delta_l}$, we have from \eqref{eq:reduced_aux1} that
\begin{align*}
0=&J_\eps'(V_1^{\mf{d},\eps},V_2^{\mf{d},\eps})[\partial_{\delta_l} V_1^{\mf{d},\eps},\partial_{\delta_l} V_2^{\mf{d},\eps}]\\
	=&J'_\eps (V_1^{\mf{d},\eps},V_2^{\mf{d},\eps})[\sum_{j\in I_1} P_\eps \partial_{\delta_l} U_{\delta_j}+\partial_{\delta_l} \phi_1^{\mf{d},\eps},\sum_{j\in I_2} P_\eps \partial_{\delta_l} U_{\delta_l} + \partial_{\delta_l} \phi_2^{\mf{d},\eps}]\\
	=& \partial_{\delta_l} J_\eps(V_1^{\mf{d},\eps},V_2^{\mf{d},\eps})[P_\eps \psi_l]+J'_\eps(V_1^{\mf{d},\eps},V_2^{\mf{d},\eps})[\partial_{\delta_l} \phi_1^{\mf{d},\eps},\partial_{\delta_l} \phi_2^{\mf{d},\eps}]\\
	=& \langle V_i^{\mf{d},\eps}-\Ical^*(f(V_i^{\mf{d},\eps})+\beta V_i^{\mf{d},\eps} \sum_{j\neq i} (V_j^{\mf{d},\eps})^2),P_\eps \psi_l\rangle \\
	&+\sum_{n=1}^2 \langle V_{n}^{\mf{d},\eps}-\Ical^*(f(V_n^{\mf{d},\eps})+\beta V_n^{\mf{d},\eps} \sum_{m\neq n} (V_m^{\mf{d},\eps})^2),\partial_{\delta_l} \phi_n^{\mf{d},\eps}\rangle 
\end{align*}
From \eqref{eq:systemajointProjection2}, Proposition \ref{prop:C^1phi} and recalling that $K_i$ is spanned by $P_\eps \psi_j$ for $j\in I_i$, we deduce the existence of coefficients $c_i^j=c_i^j(\eps,\mf{d})$, $j\in I_i$ such that
\begin{equation}\label{eq:reduced_aux2}
V_i^{\mf{d},\eps}-\Ical^*( f(V_i^{\mf{d},\eps})+\beta V_i^{\mf{d},\eps} \sum_{j\neq i}(V_j^{\mf{d},\eps})^2 ) =\sum_{j\in I_i}c_i^j \delta_j P_\eps \psi_j,\qquad i=1,2.
\end{equation}
In conclusion, for $i\in \{1,2\}$ and $l\in I_i$,
\[
\sum_{j \in I_i} c_{i}^j \langle \delta_j P_\eps \psi_j, \delta_l P_\eps \psi_l\rangle +\sum_{n=1}^2 \sum_{j\in I_n} c_n^j \langle \delta_j P_\eps \psi_j, \delta_l\partial_{\delta_l} \phi_n^{\mf{d},\eps}\rangle =0
\]
A straightforward computation shows that 
\begin{equation}\label{eq:reduced_aux3}
\langle \delta_j P \psi_j, \delta_l P \psi_l\rangle=\textrm{o}(1) \text{ for $l\neq j$}, \qquad \langle \delta_lP \psi_l, \delta_lP \psi_l\rangle=\| \delta_lP\psi_l\|^2=\sigma_{ll}+\textrm{o}(1)\qquad \text{ as } \eps\to 0
\end{equation}
for some constant $\sigma_{ll}>0$ (see for instance \cite[p. 417]{PistoiaTavares}). On the other hand,  we have $\langle P_\eps \psi_j,\partial_{\delta_l} \phi_n^{\mf{d},\eps}\rangle=\textrm{o}(1)$.  Indeed, since $\phi_n^{\mf{d},\eps}\in K_n^\perp$ $(n=1,2)$, then $\langle P_\eps \psi_j,\phi_n^{\mf{d},\eps}\rangle =0$ for every $\mf{d}$. Therefore, taking the derivative of the previous identity with respect to $\delta_l$ ($l\in I_n$), we get $\langle P_\eps \psi_j,\partial\delta_l \phi_n^{\mf{d},\eps}\rangle=-\langle \partial_{\delta_l} P_\eps \psi_j,\phi_n^{\mf{d},\eps}\rangle$. Combining \eqref{pa U con U} with Lemma \ref{lem: interaction 0} we have $\|\delta_l\partial_{\delta_l}P_\eps \phi_j\|=\textrm{O}(1)$, while Proposition \ref{prop:C^1phi} yields $\|\phi_n^{\mf{d},\eps}\|=\textrm{o}(\delta_1)$. Therefore, $\langle P_\eps \psi_j,\delta_l\partial\delta_l \phi_n^{\mf{d},\eps}\rangle=o(\delta_1)=o(1)$, as claimed. In conclusion, we end up with a linear system of the form
\[
c_i^l \sigma_{ll} + \sum_{j\in I_i\setminus\{l\}}c_i^j \textrm{o}(1) + \sum_{n=1}^2 \sum_{j\in I_n} c_n^j \textrm{o}(1)=0,\qquad i=1,2,\ l\in I_i
\] 
which, as $\eps\to 0$, has the unique solution $c_i^j=0$ for every $i=1,2$, $j\in I_i$. Looking back at \eqref{eq:reduced_aux2} we see that $(V_1^{\mf{d},\eps},V_2^{\mf{d},\eps})$ solves \eqref{eq:systemwithf}, as we wanted.
\end{proof}

We now compute the leading term of the reduced energy. For simplicity, and when there is no risk of confusion, we denote $\phi_i=\phi_i^{\mf{d},\eps}$ and $U_i=U_{\delta_i}$. We have
\begin{align}
\widetilde J_\eps (\mf{d})&=J_\eps\left(\sum\limits_{j\in I_1}P_{\eps}U_{j}+\phi_1,\sum\limits_{j\in I_2}P_{\eps}U_{j}+\phi_2\right)\\
					&= \sum_{i=1}^2 \int_{\Omega_\eps} \left(\frac{1}{2}|\nabla (\sum_{j\in I_i}P_\eps U_j+\phi_i)|^2-F(\sum_{j\in I_i}P_\eps U_j+\phi_i)\right)-\frac{\beta}{2}\int_{\Omega_\eps} (\sum_{j\in I_1}P_\eps U_j+\phi_1)^2(\sum_{j\in I_2}P_\eps U_j+\phi_2)^2\\
					&=\sum_{i=1}^2 \sum_{j\in I_i}\int_{\Omega_\eps}\left(\frac{1}{2}|\nabla P_\eps U_j|^2-F(P_\eps U_j)\right) -\frac{\beta}{2}\sum_{i\in I_1,j\in I_2} \int_{\Omega_\eps} (P_\eps U_i)^2(P_\eps U_j)^2+ R(\mf{d},\eps)\\
					&=\sum_{i=1}^k \int_{\Omega_\eps} \left(\frac{1}{2}|\nabla P_\eps U_i|^2-\frac{1}{4}(P_\eps U_i)^4\right) - \frac{\beta}{2}\sum_{i\in I_1,j\in I_2} \int_{\Omega_\eps} (P_\eps U_i)^2(P_\eps U_j)^2+ R(\mf{d},\eps), \label{eq:reducedenergy}
\end{align}
where  
\[
R(\mf{d},\eps) = J_\eps\left(\sum\limits_{j\in I_1}P_{\eps}U_{j}+\phi_1,\sum\limits_{j\in I_2}P_{\eps}U_{j}+\phi_2\right) - \sum_{i\in I_1,j\in I_2} J_\eps\left(P_\eps U_i,P_\eps U_j\right)
\]
will be an higher order term.

In what follows we show that the reduced energy reads as 
\begin{equation}\label{rid-gen}
\tilde J_\epsilon(\mathbf{d})=c_1+c_2\delta_1^2+c_3\({\epsilon \over \delta_k}\)^2 -\beta c_4\sum\limits_{i=1}^{k-1}\({\delta_{i+1}\over\delta_i}\)^2 \log{\frac{\delta_i}{\delta_{i+1}}}  +h.o.t.,
\end{equation}
for some constants $c_1,c_2,c_3, c_4>0$. This yields the choice of parameters \eqref{eq:rates} (which for convenience of the reader we recall)
\[
\delta_j:=d_j\epsilon^{j\over k+1}\(\log{1\over\epsilon}\)^{{1\over2}-{j\over k+1}}\ \hbox{with}\ d_j>0\ \hbox{for}\  j=1,\dots,k. 
\]
and the existence of towers of bubbles as we want. Observe that, as $\eps\to 0$,
\[
\delta_1^2\sim \left(\frac{\eps}{\delta_k}\right)^2 \sim \left(\frac{\delta_{i+1}}{\delta_i}\right)^2 \log \frac{\delta_i}{\delta_{i+1}}\sim \eps^{\frac{2}{k+1}}\left(\log \frac{1}{\eps} \right)^\frac{k-1}{k+1}
\]
(see ahead for the details).

\begin{lemma}\label{lemma:singleterms} 
Given $i=1,\ldots, k$ we have
\[
\int_{\Omega_\eps} \left(\frac{1}{2}|\nabla P_\eps U_i|^2-\frac{1}{4}(P_\eps U_i)^4\right)=  \frac{B}{4}+\frac{A^2}{2} \tau(0)\delta_i^2 +\frac{\Gamma}{2}\left(\frac{\eps}{\delta_i}\right)^2+\textrm{o}(\delta_i^2)+\textrm{o}\left(\left(\frac{\eps}{\delta_i}\right)^2\right)
\]
as $\eps\to 0$, uniformly for every $\mf{d}\in X_\eta$. We recall that $ A, B$ and $\Gamma$ are defined in \eqref{eq:constantsAB}--\eqref{eq:constantsAB2}, while $\tau$ is the Robin function (see the notation section).  
 \end{lemma}  
 \begin{proof}
We reason similarly to \cite[Lemma 4.3]{PistoiaSoave}, to which we refer for more details.

First of all, using \eqref{eq:Projection}, we have that
\begin{align}
\frac{1}{2} \int\limits_{\Omega_\eps} |\nabla P_\eps U_i|^2-\frac{1}{4}\int\limits_{\Omega_\eps}(P_\eps U_i)^4=&\frac{1}{2}\int\limits_{\Omega_\eps} (P_\eps U_i)U_i^3-\frac{1}{4}\int\limits_{\Omega_\eps} (P_\eps U_i)^4 \nonumber \\
=&\frac{1}{4}\int\limits_{\Omega_\eps} U_i^4+\frac{1}{2}\int\limits_{\Omega_\eps} (P_\eps U_i-U_i)U_i^3-\frac{1}{4}\int\limits_{\Omega_\eps}((P_\eps U_i)^4-U_i^4) \label{eq:firstlemma_aux}
\end{align}
Using a Taylor expansion up to second order, we have that
\[
(P_\eps U_i)^4-U_i^4= 4U_i^3 (P_\eps U_i-U_i) + 6 (U_i^2 + \xi(P_\eps U_i-U_i))^2 (P_\eps U_i-U_i)^2,
\]
for some function $\xi(x)\in [0,1]$. Therefore, we can rewrite \eqref{eq:firstlemma_aux} as 
\begin{align}\label{eq:firstlemma_aux2}
\frac{1}{4}\int\limits_{\Omega_\eps} U_i^4-\frac{1}{2}\int\limits_{\Omega_\eps} (P_\eps U_i-U_i)U_i^3-\frac{3}{2}\int\limits_{\Omega_\eps} (U_i+\xi(P_\eps U_i-U_i))^2(P_\eps U_i-U_i)^2.
\end{align}
We now estimate each one of the three terms separately. The first term in \eqref{eq:firstlemma_aux2} is, after a change of variables $x=\delta_i y$ and recalling that $\Omega_\eps=\Omega\setminus B_\eps$ and $\eps/\delta_i\to 0$,
\begin{align}
\int\limits_{\Omega_\eps}U_i^4&=\int\limits_{\Omega_\eps} \frac{\alpha_4^4 \delta_i^4}{(\delta_i^2+|x|^2)^4}\, dx= \int\limits_{\Omega_\eps\setminus \delta_i} \frac{\alpha_4^4}{(1+|y|^2)^4}\, dy\\
	&=  B + \int\limits_{\R^N\setminus (\Omega /\delta_i) } \frac{\alpha_4^4}{(1+|y|^2)^4}\, dy +\int\limits_{B_{\eps/\delta_i}} \frac{\alpha_4^4}{(1+|y|^2)^4}\, dy =B +\textrm{O}\left(\delta_i^4\right) + \textrm{O}\left(\left(\frac{\eps}{\delta_i}\right)^4\right) \\
	&=B+\textrm{o}(\delta_i^2)+\textrm{o}\left(\left(\frac{\eps}{\delta_i}\right)^2\right) \label{eq:firstlemma_aux3}
\end{align}

As for the second term, we use the fact that 
\[
P_\eps U_i- U_i = -A \delta_i H(x,0)-\frac{\alpha_4}{\delta_i} \frac{\eps^2}{|x|^2}+R(x)
\] 
(by Lemma \ref{lemma:estimates}, which we can apply since $\eps/\delta_i\to 0$ as $\eps\to 0$). We have
\begin{align}
\int\limits_{\Omega_\eps} (P_\eps U_i-U_i)U_i^3 &=  \int\limits_{\Omega_\eps} (-A\delta_iH(x,0) -\frac{\alpha_4 \eps^2}{\delta_i|x|^2})    U_i^3  + \int\limits_{\Omega_\eps} R(x) U_i^3\\
								&= \int\limits_{\Omega_\eps} \frac{\alpha_4^3\delta_i^3}{(\delta_i^2+|x|^2)^3}(-A\delta_iH(x,0) -\frac{\alpha_4 \eps^2}{\delta_i|x|^2} )+ \int\limits_{\Omega_\eps} R(x) U_i^3\\
								&= -A\int\limits_{\Omega_\eps/\delta_i} \frac{\alpha_4^3 \delta_i^2}{(1+|y|^2)^3}H(\delta_i y,0)-\int\limits_{\Omega_\eps/\delta_i} \left(\frac{\eps}{\delta_i}\right)^2 \frac{\alpha_4^4}{|y|^2(1+|y|^2)^3}+ \int\limits_{\Omega_\eps} R(x) U_i^3\\
								&=- A^2 \tau(0) \delta_i^2   -\Gamma \left(\frac{\eps}{\delta_i}\right)^2+\textrm{o}(\delta_i^2)+\textrm{o}\left(\left(\frac{\eps}{\delta_i}\right)^2\right) \label{eq:firstlemma_aux4}
\end{align}
where we have used the estimates for the remainder term $R$ contained in Lemma \ref{lemma:estimates}.

As for the last term in \eqref{eq:firstlemma_aux2}, since $0\leq P_{\eps}U_i\leq U_i$ (by the maximum principle) and $\xi(x)\in [0,1]$, we have $0\leq U_i+\xi(PU_i-U_i)\leq U_i$. Combining this with Lemma \ref{lemma:estimates2} and since $\eps/\delta_i\to 0$,
\begin{align}
\left|\, \int\limits_{\Omega_\eps}  (U_i+\xi(PU_i-U_i))^2(PU_i-U_i)^2\right| &\leq \int\limits_{\Omega_\eps} U_i^2(PU_i-U_i)^2=\textrm{O}\left(\delta_i^4|\log \delta_i|+\left(\frac{\eps}{\delta_i}\right)^4 \left|\log\left(\frac{\eps}{\delta_i}\right)\right|\right)\\
			&=\textrm{o}(\delta_i^2)+\textrm{o}\left(\left(\frac{\eps}{\delta_i}\right)^2\right). \label{eq:firstlemma_aux5}
\end{align}
The result follows combining \eqref{eq:firstlemma_aux2} with \eqref{eq:firstlemma_aux3}--\eqref{eq:firstlemma_aux4}--\eqref{eq:firstlemma_aux5}.
 \end{proof}

\begin{corollary}\label{coro:singleeq_estimate}
The following estimate holds
\begin{multline}
\sum_{i=1}^k \int_{\Omega_\eps} \frac{1}{2}|\nabla P_\eps U_i|^2-\frac{1}{4}(P_\eps U_i)^4\, dx = k \frac{B}{4}+\frac{A^2}{2} \tau(0)\delta_1^2 +\frac{\Gamma}{2}\left(\frac{\eps}{\delta_k}\right)^2+\textrm{o}(\delta_1^2)+\textrm{o}\left(\left(\frac{\eps}{\delta_k}\right)^2\right)\\
			= k \frac{B}{4}+\(\frac{A^2}{2} \tau(0)d_1^2 +\frac{\Gamma}{2}\left(\frac{1}{d_k}\right)^2\)\eps^\frac{2}{k+1}\left(\log \frac{1}{\eps}\right)^\frac{k-1}{k+1}  + \textrm{o}\left(\eps^\frac{2}{k+1}\left(\log \frac{1}{\eps}\right)^\frac{k-1}{(k+1)}\right)
\end{multline}
as $\eps\to 0$, uniformly for every $\mf{d}\in X_\eta$. 
\end{corollary}
\begin{proof}
From the previous lemma we see that
\[
\sum_{i=1}^k\int_{\Omega_\eps} \frac{1}{2}|\nabla P_\eps U_i|^2-\frac{1}{4}(P_\eps U_i)^4\, dx=  k\frac{B}{4}+\frac{A^2}{2} \tau(0)\sum_{i=1}^k\delta_i^2 +\frac{\Gamma}{2} \sum_{i=1}^k\left(\frac{\eps}{\delta_i}\right)^2+\sum_{i=1}^k\left(\textrm{o}(\delta_i^2)+\textrm{o}\left(\left(\frac{\eps}{\delta_i}\right)^2\right)\right)
\]
and the first identity of the lemma follows because $d\in N_\eta$ and $\delta_i=o(\delta_1)$ for every $i\in \{2,\ldots,k\}$ (recall Remark \ref{rem:ratesrelations}), which implies that $\eps^2/\delta_i^2=\textrm{o}(\eps^2/\delta_k^2)$ for $i\in \{1,\ldots, k-1\}$ as $\eps\to 0$.  The second identity follows directly from the definition of $d_i$ (see \eqref{eq:rates}).
\end{proof}

\begin{lemma}\label{lemma:interaction_estimate} Given $i,j\in \{1,\ldots, k\}$ with $i>j$, we have
\[
\int_{\Omega_\eps} (P_\eps U_i)^2(P_\eps U_j)^2=\alpha_4^4|\mathbb{S}^{3}| \({\delta_i\over\delta_j}\)^2 \log{\delta_j\over\delta_i} +\textrm{o}\(\({\delta_i\over\delta_j}\)^2 \log{\delta_j\over\delta_i}\)
\]
as $\eps\to 0$, uniformly for every $\mf{d}\in X_\eta$. 
\end{lemma}

\begin{proof}
First, we rewrite
\begin{equation}\label{eq:add_subtract_interactionterm}
\int\limits_{\Omega_\eps} (P_\eps U_i)^2(P_\eps U_j)^2=\int\limits_{\Omega_\eps} U_i^2U_j^2 + \int\limits_{\Omega_\eps} \left( (P_\eps U_i)^2(P_\eps U_j)^2-  ( U_i)^2(U_j)^2\right)=\int\limits_{\Omega_\eps} U_i^2U_j^2 + h.o.t.
\end{equation}
(by Lemma \ref{lemma:estimates}). We estimate the leading term as follows: for $\eps>0$ small and $r>0$ such that $B_\eps\subset B_r\subset \Omega$ and $\eps<\sqrt{\delta_i\delta_j}$,
\begin{align*}
\int\limits_{\Omega_\epsilon}   U_{i} ^{  2}   U_{j} ^{ 2} =\underbrace{\int\limits_{\{\epsilon\leq |x|\le\sqrt {\delta_i\delta_j}\}}U_{i} ^{  2}   U_{j} ^{ 2}}_{(I)}+\underbrace{\int\limits_{\{\sqrt {\delta_i\delta_j}\le |x|\le r\}} U_{i} ^{  2}   U_{j} ^{ 2}}_{(II)}+\underbrace{\int\limits_{\Omega \setminus B_r} U_{i} ^{  2}   U_{j} ^{ 2}}_{(III)}.\end{align*}

\noindent Asymptotic estimate of $(I)$: scaling $x=\delta_i y$,
\begin{align}(I)=\int\limits_{\{\epsilon\le|x|\le\sqrt {\delta_i\delta_j}\}}U_{i} ^{  2}   U_{j} ^{ 2}&=\alpha_4^4\delta_i^2\delta_j^2\int\limits_{\{\epsilon\le|x|\le\sqrt {\delta_i\delta_j}\}}{1\over	\(\delta_i^2+|x|^2\)^2}{1\over\(\delta_j^2+|x|^2\)^2}dx\\
&=\alpha_4^4\delta_i^2\delta_j^2\int\limits_{\left\{{\epsilon\over\delta_i}\le|y|\le\sqrt {\delta_j\over\delta_i}\right\}}{1\over	\(1+|y|^2\)^2}{1\over\(\delta_j^2+\delta_i^2|y|^2\)^2}dy\\
&=\alpha_4^4 \({\delta_i\over\delta_j}\)^2\int\limits_{\left\{{\epsilon\over\delta_i}\le|y|\le\sqrt {\delta_j\over\delta_i}\right\}}{1\over	\(1+|y|^2\)^2}{1\over\(1+(\delta_i/\delta_j)^2  |y|^2\)^2}dy \label{eq:computationofI}
\end{align}
We have:
\begin{align*}
&\int\limits_{\left\{{\epsilon\over\delta_i} \le|y|\le\sqrt {\delta_j\over\delta_i}\right\}} {1\over\(1+|y|^2\)^2}{1\over\(1+(\delta_i/\delta_j)^2  |y|^2\)^2}dy\\
&=\underbrace{\int\limits_{\left\{{\epsilon\over\delta_i}\le|y|\le\sqrt {\delta_j\over\delta_i}\right\}}{1\over	\(1+|y|^2\)^2}dy}_{(I.a)}+\underbrace{\int\limits_{\left\{{\epsilon\over\delta_i}\le|y|\le\sqrt {\delta_j\over\delta_i}\right\}}{1\over	\(1+|y|^2\)^2}\({1\over\(1+(\delta_i/\delta_j)^2  |y|^2\)^2}-1\)dy}_{(I.b)}. 
\end{align*}
The first term can be estimated as follows:

\begin{equation}\label{eq:asymptotic}
\begin{aligned}
(I.a) &=|\mathbb S^ {3}|\int\limits_ {\epsilon\over\delta_i}^{\sqrt {\delta_j\over\delta_i}}{r^3\over	\(1+r^2\)^2}dr=\frac12|\mathbb S^ {3}|
\left[{1\over  1+r^2 }+\log(1+r^2)\right]_{r={\epsilon\over\delta_i}}
^{r=\sqrt {\delta_j\over\delta_i}} \\
&=\frac{1}{2} |\mathbb S^{3}|\(\frac{1}{1+\delta_j/\delta_i}-\frac{1}{1+\eps^2/\delta_i^2} + \log \(1+\frac{\delta_j}{\delta_i}\)-\log \(1+\frac{\eps^2}{\delta_i^2}\)  \)\\
&=\frac12|\mathbb S^ {3}|\log\frac{\delta_j}{\delta_i}+o\(\log \frac{\delta_j}{\delta_i}\),
\end{aligned}
\end{equation}
since, as $\delta_j/\delta_i\to \infty$ ($i>j$) and $\eps/\delta_i\to 0$ (recall Remark \ref{rem:ratesrelations}):
\[
\frac{1}{1+\delta_j/\delta_i}-\frac{1}{1+\eps^2/\delta_i^2}= -1 +  \textrm{o}(1) =o\(\log \frac{\delta_j}{\delta_i}\) ,\qquad \log \(1+\delta_j/\delta_i\)=\log\frac{\delta_j}{\delta_i} +\textrm{o}\(\log\frac{\delta_j}{\delta_i}\)
\]
and
\[
 \log (1+\eps^2/\delta_i^2)=\textrm{o}(1)=o\(\log \frac{\delta_j}{\delta_i}\).
 \]
As for the second term,  because $(\delta_i/\delta_j)   |y|\le \sqrt{\delta_i/\delta_j} \leq c$ for $|y|\le\sqrt{\delta_j /\delta_i}$, as $\eps\to 0$, and recalling the computation done for $(I.a)$, we have
$$\begin{aligned}
\left| (I.b)  \right| &=   \int\limits_{\left\{{\epsilon\over\delta_i}\le|y|\le\sqrt {\delta_j\over\delta_i}\right\}}{1\over	\(1+|y|^2\)^2}\(|(\delta_i/\delta_j)^4  |y|^4+2(\delta_i/\delta_j)^2  |y|^2\over\(1+(\delta_i/\delta_j)^2  |y|^2\)^2 \)dy \\ 
& \le  c'\frac{\delta_i}{\delta_j}\int\limits_{\left\{{\epsilon\over\delta_i}\le|y|\le\sqrt {\delta_j\over\delta_i}\right\}}{1 \over	\(1+|y|^2\)^2} dy\  \le c''\frac{\delta_i}{\delta_j}\log \frac{\delta_j}{\delta_i}  =o\(\log \frac{\delta_j}{\delta_i}\).
\end{aligned}
$$
Combining the expansions of $(I.a)$ and $(I.b)$ with \eqref{eq:computationofI} yields, in conclusion, that
\begin{equation*}
(I)= \frac{\alpha_4^4}{2}|\mathbb{S}^{3}|  \({\delta_i\over\delta_j}\)^2 \log\frac{\delta_j}{\delta_i} + \textrm{o}\left( \({\delta_i\over\delta_j}\)^2 \log\frac{\delta_j}{\delta_i}\right).
\end{equation*}

\noindent Asymptotic estimate of $(II)$:  by using this time the  scaling $x=\delta_j y$ and the fact that
\[
\int \frac{1}{r(1+r^2)}\, dr= \log r -\frac{1}{2}\log(1+r^2) + \frac{1}{2(1+r^2)},
\]
we have
\begin{align*}(II)=\int\limits_{\{\sqrt {\delta_i\delta_j}\le  |x| \le r\}}U_{i} ^{  2}   U_{j} ^{ 2}&= \alpha_4^4\delta_i^2\delta_j^2\int\limits_{\{\sqrt {\delta_i\delta_j}\leq |x|\leq r\}}{1\over	\(\delta_i^2+|x|^2\)^2}{1\over\(\delta_j^2+|x|^2\)^2}dx\\
&=\alpha_4^4\delta_i^2\delta_j^2\int\limits_{\left\{\sqrt {\delta_i\over\delta_j}\le |y|\le {r\over\delta_j}\right\}}{1\over\(\delta_i^2+\delta_j^2|y|^2\)^2} {1\over	\(1+|y|^2\)^2}dy\\
&=\alpha_4^4 \({\delta_i\over\delta_j}\)^2\int\limits_{\left\{\sqrt {\delta_i\over\delta_j}\le |y|\le {r\over\delta_j}\right\}}{1\over\( (\delta_i/\delta_j)^2+  |y|^2\)^2}{1\over	\(1+|y|^2\)^2}dy\\
&=\alpha_4^4 \({\delta_i\over\delta_j}\)^2\int\limits_{\left\{\sqrt {\delta_i\over\delta_j}\le |y|\le {r\over\delta_j}\right\}}{1\over	|y|^4\(1+|y|^2\)^2}dy + \textrm{o}\(\({\delta_i\over\delta_j}\)^2 \log\frac{\delta_j}{\delta_i}\)\\
&=\alpha_4^4 \({\delta_i\over\delta_j}\)^2  |\mathbb{S}^3| \int\limits_{\sqrt {\delta_i\over\delta_j}}^{r\over\delta_j} \frac{1}{r(1+r^2)^2}\, dr + \textrm{o}\(\({\delta_i\over\delta_j}\)^2 \log\frac{\delta_j}{\delta_i}\)
\\
&= \frac{\alpha_4^4}{2}|\mathbb{S}^{3}|  \({\delta_i\over\delta_j}\)^2 \log\frac{\delta_j}{\delta_i}+ \textrm{o}\(\({\delta_i\over\delta_j}\)^2 \log\frac{\delta_j}{\delta_i}\).
\end{align*}

\noindent Asymptotic estimate of $(III)$:
\begin{align*}
0\leq \int\limits_{\Omega \setminus B_r} U_{i} ^{  2}   U_{j} ^{ 2} &=\alpha_4^4 \delta_i^2 \delta_j^2 \int\limits_{\Omega\setminus B_r} \frac{1}{(\delta_i^2+|y|^2)^2}\frac{1}{(\delta_j^2+|y|^2)^2}=\alpha_4^4\({\delta_i \over \delta_j}\)^2 \int\limits_{\Omega\setminus B_r} \frac{1}{(\delta_i^2+|y|^2)^2} \frac{1}{(1+\delta_j^{-2} |y|^2)^2} \\
				&\leq \alpha_4^4\({\delta_i \over \delta_j}\)^2 \int\limits_{\Omega\setminus B_r} \frac{1}{|y|^4} \leq c \({\delta_i \over \delta_j}\)^2= \textrm{o}\(\({\delta_i\over\delta_j}\)^2 \log\frac{\delta_j}{\delta_i}\).
\end{align*}
By combining the estimates of $(I)$, $(II)$ and $(III)$ we deduce that
\[
\int\limits_{\Omega_\epsilon}   U_{i} ^{  2}   U_{j} ^{ 2}=\alpha_4^4|\mathbb{S}^{3}| \({\delta_i\over\delta_j}\)^2 \log{\delta_j\over\delta_i} +\textrm{o}\(\({\delta_i\over\delta_j}\)^2 \log{\delta_j\over\delta_i}\)
\]
which yields the desired conclusion.
 \end{proof}

\begin{corollary}\label{coro:interaction_estimate}
We have, as $\eps\to 0$, uniformly for every $\mf{d}\in X_\eta$,
\begin{align}
\sum_{i\in I_1,j\in I_2}\, \int\limits_{\Omega_\eps} (P_\eps U_i)^2(P_\eps U_j)^2  &=\alpha_4^4|\mathbb{S}^{3}|  \sum_{i=1}^{k-1} \({\delta_{i+1}\over\delta_i}\)^2 \log{\delta_i\over\delta_{i+1}} +  \sum_{i=1}^{k-1} \textrm{o}\(\({\delta_{i+1}\over\delta_i}\)^2 \log{\delta_i\over\delta_{i+1}} \)  \nonumber\\
&=\frac{\alpha_4^4}{k+1}|\mathbb{S}^{3}|  \sum_{i=1}^{k-1}\(\frac{d_{i+1}}{d_{i}}\)^2 \eps^\frac{2}{k+1}\(\log \frac{1}{\eps}\)^{\frac{k-1}{k+1}} +\textrm{o}\left(\eps^\frac{1}{k+1}\left(\log \frac{1}{\eps}\right)^\frac{k-1}{(k+1)}\right) \label{eq:estimate_mixedterms}
\end{align}
\end{corollary}
\begin{proof}
The first identity is a simple consequence of the previous lemma together with the fact that $\delta_{l}=\textrm{o}(\delta_i)$ as $\eps\to 0$, for $l>i$. In fact, since each one of the sets $I_1$ and $I_2$ do not contain two consecutive integers, and that given $i>j$ with $|i-j|>1$ it holds
\[
\left(\frac{\delta_i}{\delta_j}\right)^2\log \frac{\delta_j}{\delta_i}=o\left( \left(\frac{\delta_{j+1}}{\delta_j}\right)^2\log \frac{\delta_j}{\delta_{j+1}} \right),
\]
then
\begin{align}
\sum_{i\in I_1,j\in I_2}\, \int\limits_{\Omega_\eps} (P_\eps U_i)^2(P_\eps U_j)^2  &=\sum_{i\in I_1,j\in I_2} \left( \alpha_4^4|\mathbb{S}^{3}| \({\delta_i\over\delta_j}\)^2 \log{\delta_j\over\delta_i} +\textrm{o}\(\({\delta_i\over\delta_j}\)^2 \log{\delta_j\over\delta_i}\)\right)\\
				&=\alpha_4^4|\mathbb{S}^{3}|  \sum_{i=1}^{k-1} \({\delta_{i+1}\over\delta_i}\)^2 \log{\delta_i\over\delta_{i+1}} +  \sum_{i=1}^{k-1} \textrm{o}\(\({\delta_{i+1}\over\delta_i}\)^2 \log{\delta_i\over\delta_{i+1}} \).
\end{align}
The last identity of the statement is a consequence of the definition of $\delta_i$ and the fact that
\begin{align*}
\(\delta_{i+1} \over \delta_i\)^2 \log {\delta_{i} \over \delta_{i+1}} &=\(\frac{d_{i+1}}{d_{i}}\)^2 \eps^\frac{2}{k+1}\(\log \frac{1}{\eps}\)^{-\frac{2}{k+1}}\log\(\frac{d_i}{d_{i+1}}\left(\frac{1}{\eps}\right)^{\frac{1}{k+1}}
\(\log \frac{1}{\eps}\)^\frac{1}{k+1}\)\\
	&=\frac{1}{k+1} \(\frac{d_{i+1}}{d_{i}}\)^2 \eps^\frac{2}{k+1}\(\log \frac{1}{\eps}\)^{\frac{k-1}{k+1}} +\textrm{o}\left(\eps^\frac{2}{k+1}\left(\log \frac{1}{\eps}\right)^\frac{k-1}{k+1}\right)   \qedhere
\end{align*} 
\end{proof}

\begin{lemma}\label{lemma:remainder} We have
\begin{equation}\label{eq:remainder}
R(\mf{d},\eps)=\textrm{o}(\delta_1^2)=\textrm{o}\left(\eps^\frac{2}{k+1}\left(\log \frac{1}{\eps}\right)^\frac{k-1}{k+1}\right)
\end{equation}
as $\eps\to 0$, uniformly for every $\mf{d}\in X_\eta$.
\end{lemma}
\begin{proof}

Recall that $F(s)=(s^+)^4/4$, and we denote $f(s):=F'(s)=(s^+)^3$. We have
\begin{align*}
				R(\mf{d},\eps)= & J_\eps\left(\sum\limits_{j\in I_1}P_{\eps}U_{j}+\phi_1,\sum\limits_{j\in I_2}P_{\eps}U_{j}+\phi_2\right) - \sum_{i\in I_1,j\in I_2} J_\eps\left(P_\eps U_i,P_\eps U_j\right)			\\
			=&\frac{1}{2}\sum_{i=1}^2  \mathop{\sum_{j,k\in I_i }}_{ j\neq k} \int_{\Omega_\eps} \nabla P_\eps U_j \cdot \nabla P_\eps U_k +  \frac{1}{2}\sum_{i=1}^2\int_{\Omega_\eps} |\nabla \phi_i|^2  + \sum_{i=1}^2 \sum_{j\in I_i} \int_{\Omega_\eps} \nabla P_\eps U_j \cdot \nabla \phi_i \\
			&+ \sum_{i=1}^2 \int_{\Omega_\eps} \left( \sum_{j\in I_i} F(P_\eps U_j)-F(\sum_{j\in I_i} P_\eps U_j+\phi_i) \right)\\
			&+\frac{\beta}{2} \sum_{i\in I_1,j\in I_2}\int_{\Omega_\eps} (P_\eps U_i)^2(P_\eps U_j)^2 - \frac{\beta}{2}\int_{\Omega_\eps} (\sum_{j\in I_1}P_\eps U_j+\phi_1)^2(\sum_{j\in I_2}P_\eps U_j+\phi_2)^2
\end{align*}
Recalling the definition of $P_\eps$ from \eqref{eq:Projection} and adding and subtracting terms of type $F(\sum_{j\in I_i} P_\eps U_j)$ and $f(\sum_{j\in I_i} P_\eps U_j)\phi_i$, we have
\begin{align*}
\frac{1}{2}\sum_{i=1}^2 & \mathop{\sum_{j,k\in I_i }}_{ j\neq k} \int_{\Omega_\eps} \nabla P_\eps U_j \cdot \nabla P_\eps U_k +  \frac{1}{2}\sum_{i=1}^2\int_{\Omega_\eps} |\nabla \phi_i|^2  + \sum_{i=1}^2 \sum_{j\in I_i} \int_{\Omega_\eps} \nabla P_\eps U_j \cdot \nabla \phi_i \\
&+ \sum_{i=1}^2 \int_{\Omega_\eps}\left( \sum_{j\in I_i} F(P_\eps U_j)-F(\sum_{j\in I_i} P_\eps U_j+\phi_i) \right)\\
=& \frac{1}{2}\sum_{i=1}^2  \mathop{\sum_{j,k\in I_i }}_{ j\neq k} \int_{\Omega_\eps}   U^3_j  P_\eps U_k\, dx+ \sum_{i=1}^2 \int_{\Omega_\eps}\left( \sum_{j\in I_i} F(P_\eps U_j)-F(\sum_{j\in I_i} P_\eps U_j)\right)\\
			& +  \frac{1}{2}\sum_{i=1}^2\int_{\Omega_\eps} |\nabla \phi_i|^2  - \sum_{i=1}^2 \int_{\Omega_\eps} \left(F(\sum_{j\in I_i} P_\eps U_j+\phi_i)-F(\sum_{j\in I_i} P_\eps U_j) -f(\sum_{j\in I_i} P_\eps U_j)\phi_i\right) \\
			&+\sum_{i=1}^2 \int_{\Omega_\eps} \left(\sum_{j\in I_i} f( U_j )-f(\sum_{j\in I_i} P_\eps U_j) \right)\phi_i.
\end{align*}
Moreover, 
\begin{align*}
\frac{\beta}{2}\sum_{i\in I_1,j\in I_2}&\int_{\Omega_\eps} (P_\eps U_i)^2(P_\eps U_j)^2 - \frac{\beta}{2}\int_{\Omega_\eps} (\sum_{j\in I_1}P_\eps U_j+\phi_1)^2(\sum_{j\in I_2}P_\eps U_j+\phi_2)^2\\
	=&- \frac{\beta}{2}\int\limits_{\Omega_\eps} \mathop{\sum_{i,j\in I_1}}_{i\neq j} \mathop{ \sum_{k,l\in I_2}}_{k\neq l} P_\eps U_iP_\eps U_j P_\eps U_k P_\eps U_l - \frac{\beta}{2}\int\limits_{\Omega_\eps} \mathop{\sum_{i,j=1}^2}_{i\neq j}  \sum_{k,l\in I_i} P_\eps U_kP_\eps U_l \sum_{m\in I_j} P_\eps U_m \phi_j  \\
	    &-   \frac{\beta}{2} \int\limits_{\Omega_\eps} \mathop{\sum_{i,j=1}^2}_{i\neq j} \sum_{k,l\in I_i} P_\eps U_k P_\eps U_l \phi_j^2 - \frac{\beta}{2}\int\limits_{\Omega_\eps} 4\sum_{i\in I_1,j\in I_2} P_\eps U_i P_\eps U_j \phi_1\phi_2\\
		&-\frac{\beta}{2} \int\limits_{\Omega_\eps} 2\mathop{\sum_{i,j=1}^2}_{i\neq j} \sum_{k\in I_i} P_\eps U_k \phi_i \phi_j^2   -  \frac{\beta}{2}\int\limits_{\Omega_\eps} \phi_1^2\phi_2^2.
  \end{align*}
Let us rewrite $R(\mf{d},\eps)$ as 

\begin{align*}
			R(\mf{d},\eps)=&\underbrace{ \frac{1}{2}\sum_{i=1}^2  \mathop{\sum_{j,h\in I_i }}_{ j\neq h} \int_{\Omega_\eps}   U^3_j  P_\eps U_h      - \frac{\beta}{2}\int\limits_{\Omega_\eps} \mathop{\sum_{i,j\in I_1}}_{i\neq j} \mathop{ \sum_{h,l\in I_2}}_{h\neq l} P_\eps U_iP_\eps U_j P_\eps U_h P_\eps U_l  }_{:=a_1}\\
			&+\underbrace{\int\limits_{\Omega_\eps} \sum_{i=1}^2 \int_{\Omega_\eps}\left( \sum_{j\in I_i} F(P_\eps U_j)-F(\sum_{j\in I_i} P_\eps U_j)\right)}_{=:a_2}\\
			& + \underbrace{ \sum_{i=1}^2 \int_{\Omega_\eps} \left(F(\sum_{j\in I_i} P_\eps U_j)-F(\sum_{j\in I_i} P_\eps U_j+\phi_i) +f(\sum_{j\in I_i} P_\eps U_j)\phi_i\right)}_{:=a_3} \\
			&+\underbrace{  \sum_{i=1}^2 \int_{\Omega_\eps} \left(\sum_{j\in I_i} f( U_j )-f(\sum_{j\in I_i} P_\eps U_j) \right)\phi_i  - \frac{\beta}{2}\int\limits_{\Omega_\eps} \mathop{\sum_{i,j=1}^2}_{i\neq j}  2\sum_{h,l\in I_i} P_\eps U_h P_\eps U_l \sum_{m\in I_j} P_\eps U_m \phi_j }_{:=a_4}\\
		& \underbrace{-   \frac{\beta}{2} \int\limits_{\Omega_\eps} \mathop{\sum_{i,j=1}^2}_{i\neq j} \sum_{h,l\in I_i} P_\eps U_h P_\eps U_l \phi_j^2 - \frac{\beta}{2}\int\limits_{\Omega_\eps} 4\sum_{i\in I_1,j\in I_2} P_\eps U_i P_\eps U_j \phi_1\phi_2}_{:=a_5}\\
		&-\underbrace{\frac{\beta}{2} \int\limits_{\Omega_\eps} 2\mathop{\sum_{i,j=1}^2}_{i\neq j} \sum_{h\in I_i} P_\eps U_h \phi_i \phi_j^2 +  \frac{1}{2}\sum_{i=1}^2\int_{\Omega_\eps} |\nabla \phi_i|^2   -  \frac{\beta}{2}\int\limits_{\Omega_\eps} \phi_1^2\phi_2^2}_{=:a_6}.
\end{align*} 

\smallbreak

\textit{Estimates for $a_1,a_2$.} First of all, we check that the first two terms satisfy $a_1,a_2=o(\delta_1^2).$ Indeed, since $0\leq P_\eps U_i\leq U_i$ (by the maximum principle), $a_1+a_2$ is controlled by a sum of terms of the form
$  
 \int_{\Omega_\eps}P_\eps U_i P_\eps U_j  P_\eps U_h P_\eps U_l$ for   indices $j,l,h,i$ not all equal at the same time; each term is of higher order with respect to the leading term $\delta_1^2$, as we will now check.  
Indeed, if  $i\not=j$ we have by Lemma \ref{lem: interaction 1} that
\begin{align*}
\int\limits_{\Omega_\eps} U_i^3 U_j &= \begin{cases}
\displaystyle \textrm{O}\left(\frac{\delta_i}{\delta_j}\right) \int_{\Omega_\eps/\delta_i} \frac{1}{(1+|y|^2)^3}=\textrm{O}\left(\frac{\delta_i}{\delta_j}\right) & \text{ if } i>j\\[10pt]
\displaystyle \textrm{O}\left(\frac{\delta_j}{\delta_i}\right) \int_{\Omega_\eps/\delta_i} \frac{1}{(1+|y|^2)^3|y|^2}=\textrm{O}\left(\frac{\delta_j}{\delta_i}\right) & \text{ if } j>i.\\
\end{cases}
\\
  &=o(\delta_1^2)
\end{align*}
because we are always in a situation that $i,j$ belong to the same set $I_h$, thus $|i-j|>1$, and then by the choices we did in \eqref{eq:rates}, 
\[
\frac{\delta_i}{\delta_j}=\textrm{O}\left(\varepsilon^{i-j\over k+1}\left(\ln\frac1\varepsilon\right)^{-{i-j\over k+1}}\right)=o\left(\delta_1^2\right) \qquad \text{ whenever $i- j\ge 2$}.
\]
Moreover, if $i\not=j$, then assuming without loss of generality that $i>j$ with $|i-j|>1$ we have by Lemma \ref{lem: interazione 2}
\begin{align*}
\int\limits_{\Omega_\eps} U_i^2 U_j^2 =\textrm{O}\left( \frac{\delta_i^2}{\delta_j^2}|\log \delta_i|\right)=o\left(\frac{\delta_i}{\delta_j}\right)  =o(\delta_1^2).
 \end{align*}
 (note that this term only appears in $a_2$). In a similar way, if $i,j,l\in I_h$ for some $h\in \{1,2\}$, then  $|i-j|,|i-l|, |j-l|>1 $ and 
\begin{align*}
\int\limits_{\Omega_\epsilon}U_i^2 U_j  U_l dx &=  \int\limits_{\Omega_\epsilon}   {\alpha_4^2\delta_i^2\over(\delta_i^2+|x|^2)^2}{\alpha_4\delta_j \over\delta_j^2+|x|^2 }{\alpha_4\delta_l \over\delta_l^2+|x|^2 }dx= \int\limits_{\Omega_\epsilon/\delta_i}   \frac{\alpha_4^4\delta_i^2\delta_j\delta_\ell}{ (1+|y|^2)^2(\delta_j^2+\delta_i^2|y|^2)(\delta_l^2+\delta_i^2|y|^2) }dy\\ 
&=\begin{cases}  O\({\delta_j\delta_l\over \delta_i^2}|\log\delta_i|\)\ \hbox{if}\ i<j<l\\ 
 O\({\delta_ l\over \delta_j} \)\ \hbox{if}\ j<i<l\\
  O\({ \delta_i^2\over \delta_j\delta_l}|\log\delta_i|\)\ \hbox{if}\ j< l<i\\
  \end{cases}\\
  &=o(\delta_1^2)
\end{align*}
(note that this term only appears in $a_2$).Finally, if all the indices are different, then assuming without loss of generality  that $i>j>h>l$, 
\begin{align*}
\int\limits_{\Omega_\eps} U_i U_j U_h U_l \, dx&=  \int_{\Omega_\eps}\frac{\alpha_4 \delta_i}{\delta_i^2+|x|^2}\frac{\alpha_4\delta_j}{\delta_j^2+|x|^2}\frac{\alpha_4 \delta_h}{\delta_h^2+|x|^2}\frac{\alpha_4\delta_l}{\delta_l^2+|x|^2}\, dx   \\
				&\leq \int_{\Omega_{\eps}/\delta_i} \frac{\alpha_4^4\delta_i^3 \delta_j\delta_h\delta_l}{(1+|y|^2)\delta_h^2\delta_l^2 |y|^2}\, dy =\textrm{O} \left(\delta_1^2\frac{\delta_i}{\delta_h}\frac{\delta_j}{\delta_{l}} |\log \delta_i|\right) \, dy = \textrm{o}(\delta_1^2).
\end{align*}

\smallbreak

\textit{Estimates for $a_3$.} Arguing as in the proof of Lemma 7.2 in \cite{GeMuPi} (see equation (7.6) therein), the term $a_3$ is quadratic in $\phi_1$ and $\phi_2$,  and so by Proposition \ref{prop:C^1phi} it satisfies $a_2=o(\delta_1^2)$.
 
 \smallbreak
 
\textit{Estimates for $a_4$.} The first term in $a_4$ can be estimated as
\begin{align*}
|\int_{\Omega_\eps} \left(\sum_{j\in I_i} f( U_j )-f(\sum_{j\in I_i} P_\eps U_j) \right)\phi_i  \, dx|&\le  \|\sum_{j\in I_i} f( U_j )-f(\sum_{j\in I_i} P_\eps U_j)\|_{L^{\frac43}(\Omega_\eps)}\|\phi_i\|_{L^ 4(\Omega_\eps)}\\
&=O(\delta_1^2)\|\phi_i\|_{H^1_0 (\Omega_\eps)}=o(\delta_1^2),\end{align*}
because, by \eqref{st I1 7},
  $$\|\sum_{j\in I_i} f( U_j )-f(\sum_{j\in I_i} P_\eps U_j)\|_{L^{\frac43}(\Omega_\eps)}=O(\delta_1^2).$$  
 (this term corresponds to the quantity $A$ defined in the proof of Proposition \ref{prop: stima rem}). As for the second term in $a_4$, given $i,j\in \{1,2\}$ with $i\neq j$ and $h,l\in I_i$ with $h<l$, $m\in I_j$, by Lemma \ref{lem: int a 3} we have
 \begin{align*}
 \left|\,\int\limits_{\Omega_\eps} P_\eps U_h P_\eps U_l P_\eps U_m \phi_j \right| &\leq \| P_\eps U_h P_\eps U_l P_\eps U_m \|_{L^{4/3}} \|\phi_j\|_{L^4}\leq \| U_h  U_l  U_m \|_{L^{4/3}} \|\phi_j\|_{L^4}\\
 					&= \begin{cases}
					\text{O}\left(\frac{\delta_l}{\delta_h}\right) \|\phi_j\|_{H^1_0}  & \text{ if } h<l<m  \\
					\text{O}\left(\frac{\delta_m}{\delta_k}\right)		\|\phi_j\|_{H^1_0}  & \text{ if } h<m<l\\
					\text{O}\left(\frac{\delta_h}{\delta_m}\right)		\|\phi_j\|_{H^1_0}  & \text{ if } m<h<l
					\end{cases}\\
					&=\text{o}(\delta_1^2)
 \end{align*}
 since, for instance when $h<l<m$,
 \begin{equation}\label{eq:crucial_ln}
 \text{O}\left(\frac{\delta_l}{\delta_h}\right) \|\phi_j\|_{H^1_0}\delta_1^{-2} =\text{O}(1)\eps^{l-h-1}\left(\ln \frac{1}{\eps}\right)^{-\frac{k+l-h}{h+1}}\to 0.
 \end{equation}
 
 \smallbreak
 
 \textit{Estimates for $a_5$.} Starting from the first term, by Proposition \ref{prop:C^1phi} and Lemma \ref{lem: interazione 2} we see that, given $i,j\in \{1,2\}$ with $i\neq j$, and $h,l\in I_i$ with $h<l$,
 \begin{align*}
 \left|\, \int\limits_{\Omega_\eps} P_\eps U_h P_\eps U_l \phi_j^2 \right|  \leq C \|  U_h U_l\|_{L^2} \|\phi_j\|_{H^1_0}^2=\text{o}\left(\frac{\delta_l\delta_1^2}{\delta_h}\right)=\text{o}(\delta_1^2).
 \end{align*}
 Similarly, the second term in $a_5$ is also an $\text{o}(\delta_1^2)$.
 
 \smallbreak
 
  \textit{Estimates for $a_6$.} We have, by Proposition \ref{prop:C^1phi}  and Lemma \eqref{lem: interaction 0},
  \begin{align*}
\left|\, \int\limits_{\Omega_\eps}  P_\eps U_h \phi_i \phi_j^2 \right|\leq C \|U_h\|_{L^4} \|\phib\|_{H^1_0}^3=\text{o}(\delta_1^2)
  \end{align*}
  while the second and third terms in $a_6$ are respectively of second and fourth order in $\phi$, thus an $\text{o}(\delta_1^2)$. This ends the proof.
\end{proof}

As a direct consequence of \eqref{eq:reducedenergy}, Corollary \ref{coro:singleeq_estimate}, Corollary \ref{coro:interaction_estimate} and Lemma \ref{lemma:remainder}, we have the following result, which gives us the leading term of the expansion of the reduced energy.

\begin{proposition}\label{coro:expansion}
We have
\begin{multline}
\widetilde J_\eps (\mf{d})=k \frac{B}{4}+\(\frac{A^2}{2} \tau(0)d_1^2 +\frac{\Gamma}{2}\left(\frac{1}{d_k}\right)^2-\frac{\beta}{2}\frac{\alpha_4^4}{k+1}|\mathbb{S}^{3}|  \sum_{i=1}^{k-1}\(\frac{d_{i+1}}{d_{i}}\)^2\)\eps^\frac{2}{k+1}\left(\log \frac{1}{\eps}\right)^\frac{k-1}{k+1} \\
			 + \textrm{o}\left(\eps^\frac{2}{k+1}\left(\log \frac{1}{\eps}\right)^\frac{k-1}{k+1}\right) \label{eq:expansion}
\end{multline}
as $\eps\to 0$, uniformly in $\mf{d}\in X_\eta$.
\end{proposition}

\begin{remark}\label{rem:remark_section_reduced}
Suppose that, instead of dealing with the set of odd and even numbers of $\{1,\ldots,k\}$ in the two equation case, we are dealing with  system \eqref{eq:gensystem} with $m$ equations and with a general partition $I_1,\ldots, I_m$ satisfying (1)--(5). Then the reduced energy reads as
\begin{align*}
\widetilde J_\eps (\mf{d})&=J_\eps\left(\sum\limits_{j\in I_1}P_{\eps}U_{\delta_j}+\phi_1^{\mf{d},\eps},\ldots, \sum\limits_{j\in I_m}P_{\eps}U_{\delta_j}+\phi_m^{\mf{d},\eps}\right)\\
					&=\sum_{i=1}^k \int_{\Omega_\eps} \left(\frac{1}{2}|\nabla P_\eps U_i|^2-\frac{1}{4}(P_\eps U_i)^4\right) - \frac{\beta}{2}\mathop{\sum_{h_1,h_2=1}^k}_{h_1\neq h_2} \mathop{\sum_{i\in I_{h_1}}}_{j\in I_{h_2}} \int\limits_{\Omega_\eps} (P_\eps U_i)^2(P_\eps U_j)^2+ R(\mf{d},\eps)
\end{align*}
with
\[
R(\mf{d},\eps)=J_\eps\left(\sum\limits_{j\in I_1}P_{\eps}U_{\delta_j}+\phi_1^{\mf{d},\eps},\ldots, \sum\limits_{j\in I_m}P_{\eps}U_{\delta_j}+\phi_m^{\mf{d},\eps}\right) - \mathop{\sum_{h_1,h_2=1}^k}_{h_1\neq h_2} \mathop{\sum_{i\in I_{h_1}}}_{j\in I_{h_2}} J_\eps\left(P_\eps U_i,P_\eps U_j\right)
\]
With an analogous proof of the one of Lemma \ref{lemma:equivalent}, we can show that critical points of this functional correspond to solutions of \eqref{eq:gensystem}. The choice of rates is still \eqref{eq:rates} in the general case. As pointed out in the proofs of Corollary \ref{coro:interaction_estimate} and Lemma \ref{lemma:remainder}, besides the exact shape of the rates, the other crucial step is that each set $I_h$ does not contain two consecutive integers. Since this property is valid for a general partition (it corresponds to property (5)), it is straightforward to adapt the proofs  of these results and show that the quantity
\[
\mathop{\sum_{h_1,h_2=1}^k}_{h_1\neq h_2} \mathop{\sum_{i\in I_{h_1}}}_{j\in I_{h_2}} \int\limits_{\Omega_\eps} (P_\eps U_i)^2(P_\eps U_j)^2
\]
has the asymptotic expansion \eqref{eq:estimate_mixedterms}, and that $R$ satisfies \eqref{eq:remainder}. Combining this with Corollary \ref{coro:singleeq_estimate} yields that, in the general case, the reduced expression has the exact same expansion, namely \eqref{eq:expansion}.
\end{remark}


 \section{Proof of the main result}\label{sec:proof} 

In this section we conclude the proof of Theorem \ref{thm:main2}. Define $\Psi:(\R^+)^k\to \R$ as
\[
\Psi(x_1,\ldots, x_k)=a_1 x_1 +\frac{a_2}{x_k}+a_3  \sum_{i=1}^{k-1} \frac{x_{i+1}}{x_{i}}
\]
where, since $\beta<0$,
\[
a_1=\frac{A^2}{2} \tau(0)>0,\quad a_2=\frac{\Gamma}{2}>0,\quad a_3=-\frac{\beta}{2}\frac{\alpha_4^4}{k+1}|\mathbb{S}^{3}|>0.
\]
\begin{lemma}\label{lemma:global_minimum}
The function $\Psi$ achieves a unique global minimum at $(x_1^*, \ldots, x_k^*)$, with
\[
x_i^*:= \(\frac{a_2}{a_3}\)^\frac{i}{k+1}\( \frac{a_3}{a_1}\)^\frac{k+1-i}{k+1}= \Gamma^\frac{i}{k+1}\(A^2 \tau(0)\)^\frac{i-1-k}{k+1}\(\frac{|\beta| \alpha_4^4 |\mathbb{S}^3|}{k+1}\)^\frac{k+1-2i}{k+1}>0
\]
for $i=1,\ldots, k$. In particular, the conclusion of Theorem \ref{thm:main2} holds true.
\end{lemma}
\begin{proof}
First of all, observe that
\[
\Psi(\mf{x})\to \infty \qquad \text{ as } |\mf{x}|\to \infty,
\]
since $a_1,a_2,a_3>0$ and $a_1x_1\to \infty$ if $x_1\to \infty$, $x_{i+1}/x_i\to \infty$ if $x_{i+1}\to \infty$ and $x_i$ is bounded. Moreover, given $\bar{ \mf{x}}\in (\R^+_0)^k$ with $\bar x_i=0$, 
\[
\Psi(\mf{x})\to \infty \qquad \text{ as } \mf{x}\to \bar{\mf{x}},
\]
since in this case at least one of the $k+1$ terms in the expression of $\Psi$ divergences to $\infty$. In conclusion, $\Psi$ admits a global minimum. Let us see next that it is unique, and deduce its expression.

We have
\[
\frac{\partial \Psi}{\partial x_1}=a_1-a_3\frac{x_2}{x_1^2},\qquad , \frac{\partial \Psi}{\partial x_k}=-\frac{a_2}{x_k^2}+\frac{a_3}{x_{k-1}}
\]
and
\[
\frac{\partial \Psi}{\partial x_j}=a_3 \(\frac{1}{x_{j-1}}-\frac{x_{j+1}}{x_j^2}\).
\qquad (j=2,\ldots, k-1)
\]
Hence, at a critical point,
\[
x_1^2 a_1=a_3 x_2,\quad x_j^2=x_{j-1}x_{j+1}\ \ (j=2,\ldots, k-1),\quad a_3 x_k^2=a_2x_{k-1}
\]
which yields, by direct substitution,
\[
x_j=\(\frac{a_1}{a_3}\)^{j-1} x_1^j\ \ (j=2,\ldots, k),\quad a_3\(\frac{a_1}{a_3}\)^{2k-2}x_1^{2k}= a_2 \(\frac{a_1}{a_3}\)^{k-2}x_1^{k-1}.
\]
The last identity gives 
\[
x_1=\(\frac{a_2}{a_3}\)^\frac{1}{k+1} \(\frac{a_3}{a_1}\)^\frac{k}{k+1},
\]
and the rest of the proof follows. \end{proof}

\smallbreak

\begin{proof}[End of the proof of Theorem \ref{thm:main2}] From the definition of $\Psi$ and by Proposition \ref{coro:expansion}, we have
\[
J(\mf{d})=k \frac{B}{4}+\eps^\frac{2}{k+1}\left(\log \frac{1}{\eps}\right)^\frac{k-1}{k+1}  \( \Psi(d_1^2,\ldots, d_k^2)+ \textrm{o}(1)\)
\]
where $\textrm{o}(1)\to 0$ as $\eps\to 0$, uniformly in $\mf{d}\in X_\eta$. Let $d_i^*:=\sqrt{x_i^*}$ (cf. Lemma \ref{lemma:global_minimum}), and take $\eta>0$ small enough so that $(d_1^*,\ldots, d_k^*)\in X_\eta$. Let $K\Subset X_\eta$ be a compact set such that $((d_1^*)^2,\ldots, (d_k^*)^2)\in \text{int} K$ and
\[
\Psi((d_1^*)^2,\ldots, (d_k^*)^2)=\min_K \Psi< \min_{\partial K} \Psi.
\]
Then
\[
\min_{K} J_\eps\leq J_\eps(\mf{d}^*)<\min_{\partial K} J_\eps
\]
Therefore $J_\eps|_K$ has a minimizer $\mf{d}_\eps$, which converves to $\mf{d}^*$ (by the uniqueness stated in Lemma  \ref{lemma:global_minimum}). Thus $J_\eps'(\mf{d}^*)=0$. By invoking Lemma \ref{lemma:equivalent}, the proof is finished.
\end{proof}

\begin{remark}\label{rem:section_conclusion}
The proof of the general case, Theorem \ref{thm:main}, follows exactly in the same way since, as we commented on Remark \ref{rem:remark_section_reduced}, the reduced functional is the same as in the two equation case.
\end{remark}


\appendix

\section{Asymptotic estimates}
In this appendix we collect several important asymptotic estimates which are used in the paper.We assume in every statement that $N=4$, that is, $\Omega\subset \R^4$.

The following two results are taken from \cite{GeMuPi}, see Lemmas 3.1 and 3.2 therein.

\begin{lemma}\label{lemma:estimates}
Let $a\in \Omega$, $r>0$, and $\tau\in \R^4$. Assume that $\xi=a+\delta \tau$, with $\delta=\delta(\eps)\to 0$ and $\eps/\delta \to 0$ as $\eps\to 0^+$.
Then, for $A=\int_{\R^N}U_{1,0}^3=\frac{\alpha_4}{\gamma_4}$ and $R$ defined by
\[
P_{\eps} U_{\delta,\xi}=U_{\delta,\xi}-A \delta H(x,\xi)-\frac{\alpha_4}{\delta(1+|\tau|^2)}\left(\frac{r\eps}{|x-a|}\right)^{2}+R(x)
\]
there exists $C=C(\tau,\textrm{dist}(a,\partial \Omega))>0$ such that, for any $x \in \Omega \setminus B_{r \eps}(a)$,
\begin{align*}
|R(x)|&\leq C \delta \left[\frac{\eps^{2}(1+\eps \delta^{-3})}{|x-a|^{2}}+\delta^2+\left(\frac{\eps}{\delta}\right)^{2}\right]\\
|\partial_{\delta} R(x)|&\leq C \left[ \frac{\eps^{2}(1+\eps \delta^{-3})}{|x-a|^{2}}+\delta^2+\left(\frac{\eps}{\delta}\right)^{2}\right]\\
|\partial_{\tau_i} R(x)| & \leq C \delta^2\left[ \frac{\eps^{2}(1+\eps \delta^{-4})}{|x-a|^{2}}+\delta^2+\frac{\eps^{2}}{\delta^{3}}\right]
\end{align*}
\end{lemma}

\begin{lemma}\label{lemma:estimates2}
Under the assumptions and notations of the previous lemma, we have the following estimate:
\[
\int_{\Omega_\eps}U_{\delta,\xi}^2(P_{\eps}U_{\delta,\xi}-U_{\delta,\xi})^2=
\textrm{O}\left(\delta^4 |\log \delta|+ \left(\frac{\eps}{\delta}\right)^4|\log\frac{\eps}{\delta}|\right).
\]
%
\end{lemma}

The following concerns the asymptotic study of $L^q$ norms of the bubble. For the proof, see for instance \cite[Lemma A.3]{PistoiaSoave} or \cite[Lemma A.2]{PistoiaTavares}.
\begin{lemma}\label{lem: interaction 0}
We have, as $\delta\to 0$,
\[
\int_\Omega U_\delta^q=\begin{cases}
O(\delta^q) & \text{ if } 0<q<2,\\
O(\delta^2 |\log \delta|) & \text{ if } q=2,\\
O(\delta^{4-q}) & \text{ if } 2<q<\infty,\ q\neq 4,\\
O(1) & \text{ if } q=4.
\end{cases}
\]

\end{lemma}

The following lemmas will be used many times in order to estimate interaction integrals.

\begin{lemma}\label{lem: interaction 1}
Let $1<q<2<p$ be such that $p+q=4$. Let $\rho_{1} = \rho_1(\eps)>0$, $\rho_2 =  \rho_2(\eps)>0$, be such that 
\[
\frac{\rho_2}{\rho_1} \to 0, \quad \frac{\eps}{\rho_1} \to 0, \quad \frac{\eps}{\rho_2} \to 0
\]
as $\eps \to 0^+$. Then
\[
\int_{\Omega_\eps} U_{\rho_1}^p U_{\rho_2}^q = O\left( \left(\frac{\rho_2}{\rho_1}\right)^q\right), \quad \text{and} \quad \int_{\Omega_\eps} U_{\rho_2}^p U_{\rho_1}^q = O\left( \left(\frac{\rho_2}{\rho_1}\right)^q\right).
\]
\end{lemma}
\begin{proof}
We proceed by direct computations:
\begin{align*}
\int_{\Omega_\eps} U_{\rho_1}^p U_{\rho_2}^q & = \alpha_4^4 \int_{\Omega_\eps} \left(\frac{\rho_1}{\rho_1^2 + |x|^2}\right)^p  \left(\frac{\rho_2}{\rho_2^2 + |x|^2}\right)^q\,dx \\
& = \alpha_4^4 \rho_1^{4-p}\int_{\Omega_\eps/\rho_1} \left(\frac{1}{1 + |y|^2}\right)^p  \left(\frac{\rho_2}{\rho_2^2 + \rho_1^2 |y|^2}\right)^q\,dy \\
& \le \alpha_4^4 \left(\frac{\rho_2}{\rho_1}\right)^q \int_{\Omega_\eps/\rho_1} \left(\frac{1}{1 + |y|^2}\right)^p  \frac{1}{|y|^{2q}}\,dy \\
& \le \alpha_4^4 \left(\frac{\rho_2}{\rho_1}\right)^q \int_{\R^4} \left(\frac{1}{1 + |y|^2}\right)^p  \frac{1}{|y|^{2q}}\,dy = O\left(\left(\frac{\rho_2}{\rho_1}\right)^q\right),
\end{align*}
as desired. Analogously
\begin{align*}
\int_{\Omega_\eps} U_{\rho_2}^p U_{\rho_1}^q & = \alpha_4^4 \int_{\Omega_\eps} \left(\frac{\rho_2}{\rho_2^2 + |x|^2}\right)^p  \left(\frac{\rho_1}{\rho_1^2 + |x|^2}\right)^q\,dx \\
& = \alpha_4^4 \rho_2^{4-p}\int_{\Omega_\eps/\rho_2} \left(\frac{1}{1 + |y|^2}\right)^p  \left(\frac{\rho_1}{\rho_1^2 + \rho_2^2 |y|^2}\right)^q\,dy \\
& \le \alpha_4^4 \left(\frac{\rho_2}{\rho_1}\right)^q \int_{\Omega_\eps/\rho_2} \left(\frac{1}{1 + |y|^2}\right)^p\,dy \\
& \le \alpha_4^4 \left(\frac{\rho_2}{\rho_1}\right)^q \int_{\R^4} \left(\frac{1}{1 + |y|^2}\right)^p \,dy = O\left(\left(\frac{\rho_2}{\rho_1}\right)^q\right). \qedhere
\end{align*}
\end{proof}

\begin{lemma}\label{lem: interazione 2}
Let $\rho_{1} = \rho_1(\eps)>0$, $\rho_2 =  \rho_2(\eps)>0$, be such that 
\[
\frac{\rho_2}{\rho_1} \to 0, \quad \frac{\eps}{\rho_1} \to 0, \quad \frac{\eps}{\rho_2} \to 0
\]
as $\eps \to 0^+$. Then
\[
\int_{\Omega_\eps} U_{\rho_1}^2 U_{\rho_2}^2 = O \left(\frac{\rho_2}{\rho_1} |\log \rho_2| \right) .
\]
\end{lemma}
\begin{proof}
We have
\begin{align*}
\int_{\Omega_\eps} U_{\rho_1}^2 U_{\rho_2}^2 & = \alpha_4^4 \int_{\Omega_\eps} \left(\frac{\rho_1}{\rho_1^2 + |x|^2}\right)^2  \left(\frac{\rho_2}{\rho_2^2 + |x|^2}\right)^2\,dx \\
& = \alpha_4^4 \rho_1^2 \rho_2^2 \int_{\Omega_\eps/\rho_2} \left(\frac{1}{\rho_1^2 + \rho_2^2|y|^2}\right)^2  \left(\frac{1}{1 + |y|^2}\right)^2\,dx \\
& \le \alpha_4^4 \left(\frac{\rho_2}{\rho_1}\right)^2 \int_{\Omega_\eps/\rho_2} \frac{dy}{(1+|y|^2)^2}.
\end{align*}
Since $\Omega$ is bounded, there exists a sufficiently large radius $R>0$ such that
\begin{align*}
\int_{\Omega_\eps/\rho_2} \frac{dy}{(1+|y|^2)^2} & \le \int_{B_R/\rho_2} \frac{dy}{(1+|y|^2)^2} \le C + C \int_1^{R/\rho_2} r^{-1} \, dr = C + C \log \frac{R}{\rho_2}.
\end{align*}
To sum up
\begin{align*}
\int_{\Omega_\eps} U_{\rho_1}^2 U_{\rho_2}^2 & \le \alpha_4^4 \left(\frac{\rho_2}{\rho_1}\right)^2 \left( C + C \log \frac{R}{\rho_2} \right),
\end{align*}
and the thesis follows.
\end{proof}

\begin{lemma}\label{lem: int a 3}
Let $\rho_{1} = \rho_1(\eps)>0$, $\rho_2 =  \rho_2(\eps)>0$, $\rho_3 =  \rho_3(\eps)>0$ be such that 
\[
\frac{\rho_j}{\rho_i} \to 0 \quad \text{if $1 \le i<j \le 3$}, \quad \frac{\eps}{\rho_h} \to 0 \quad \text{for every $h$},
\]
as $\eps \to 0^+$. Then
\[
\int_{\Omega_\eps} (U_{\rho_1} U_{\rho_2} U_{\rho_3})^\frac43 = O \left(\left( \frac{\rho_2}{\rho_1} \right)^\frac43\right)
\]
as $\eps \to 0^+$.
\end{lemma}

\begin{proof}
Again, by direct computations
\begin{align*}
\int_{\Omega_\eps} (U_{\rho_1} U_{\rho_2} U_{\rho_3})^\frac43 & = \alpha_4^4 \int_{\Omega_\eps} \left( \frac{\rho_1 \rho_2 \rho_3 }{(\rho_1^2 + |x|^2)^2  (\rho_2^2 + |x|^2)^2 (\rho_3^2 + |x|^2)}\right)^\frac43 \,dx \\
& = \alpha_4^4 (\rho_3^2 \rho_1 \rho_2)^\frac43 \int_{\Omega_\eps/\rho_3} \frac{dy}{(1 + |y|^2)^{\frac43} (\rho_1^2 + \rho_3^2 |y|^2 )^\frac43 (\rho_2^2 + \rho_3^2 |y|^2 )^\frac43 } \\
& \le \alpha_4^4 \left(\frac{\rho_2}{\rho_1}\right)^\frac43 \int_{\R^4} \frac{dy}{(1 + |y|^2)^{\frac43} |y|^{\frac{8}3}}. \qedhere
\end{align*} 
\end{proof}

\medbreak

\end{document}